\documentclass[preprint]{elsarticle}

\makeatletter
\def\ps@pprintTitle{%
	\let\@oddhead\@empty
	\let\@evenhead\@empty
	\def\@oddfoot{\footnotesize\itshape
		{} \hfill\today}%
	\let\@evenfoot\@oddfoot
}
\makeatother

\usepackage{latexsym}
\usepackage{lineno}
\usepackage{hyperref}
\usepackage{indentfirst}
\usepackage{amsxtra}
\usepackage{amssymb}
\usepackage{amsthm}
\usepackage{natbib}
\usepackage{amsmath}
\usepackage{tikz}
\usepackage{amscd}
\usepackage[capitalise]{cleveref}
\newtheorem{theor}{Theorem}
\newtheorem{prop}[theor]{Proposition}
\newtheorem{cor}[theor]{Corollary}
\newtheorem{lemma}[theor]{Lemma}

\theoremstyle{definition} 
\newtheorem{defin}{Definition}
\newtheorem{rem}{Remark}
\newtheorem*{conv}{Convention}
\newtheorem{rems}{Remarks}

\newtheorem{ex}{Example}
\newtheorem{exs}[ex]{Examples}

\DeclareMathOperator{\Sym}{Sym}
\DeclareMathOperator{\id}{id}
\DeclareMathOperator{\Aut}{Aut}
\DeclareMathOperator{\Ret}{Ret}

\DeclareMathOperator{\Soc}{Soc}

\begin{document}

\begin{frontmatter}
	\title{Indecomposable involutive set-theoretic solutions of the Yang-Baxter equation and orthogonal dynamical extensions of cycle sets\tnoteref{mytitlenote}}
	\tnotetext[mytitlenote]{This work was partially supported by the Dipartimento di Matematica e Fisica ``Ennio De Giorgi" - Università del Salento. The authors are members of GNSAGA (INdAM).}
	\author[unile]{Marco~CASTELLI}
	\ead{marco.castelli@unisalento.it}
	\author[unile]{Francesco~ CATINO
	}
	\ead{francesco.catino@unisalento.it}
	\author [unile] {Paola~STEFANELLI}
	\ead{paola.stefanelli@unisalento.it}
		\cortext[c1]{Corresponding author}
	\address[unile]{Dipartimento di Matematica e Fisica ``Ennio De Giorgi"
		\\
		Universit\`{a} del Salento\\
		Via Provinciale Lecce-Arnesano \\
		73100 Lecce (Italy)\\}

\begin{abstract} 
Employing the algebraic structure of the left brace and the dynamical extensions of cycle sets, we investigate a class of indecomposable involutive set-theoretic solutions of the Yang-Baxter equation having specific imprimitivity blocks. Moreover, we study one-generator left braces of multipermutation level $2$.
\end{abstract}
\begin{keyword}
\texttt{cycle set\sep set-theoretic solution\sep Yang-Baxter equation\sep brace}
\MSC[2020] 16T25\sep 20N02\sep 20E22 \sep 81R50
\end{keyword}

\end{frontmatter}

\section{Introduction}

An intriguing challenge that attracts mathematicians in recent years is 
providing a complete description of the solutions to the quantum Yang-Baxter equation, a fundamental equation of Theoretical Physics.
The interest in this field of research has 
arisen from the paper by Drinfel'd \cite{drinfeld1992some}, where the author suggested the study of set-theoretic solutions.
A \emph{set-theoretic solution of the Yang-Baxter equation} is a pair $(X,r)$ where $X$ is a non-empty set and $r$ is a map from $X\times X$ into itself satisfying the braid relation
\begin{align*}
    (r\times\id_X)(\id_X\times r)(r\times\id_X)
    =
    (\id_X\times r)(r\times\id_X)(\id_X\times r).
\end{align*}
Writing the map $r$ as $r(x,y) = (\lambda_x(y), \rho_y(x))$, for all $x,y\in X$, where $\lambda_x$ and $\rho_y$ are maps from $X$ into itself, a solution $(X, r)$ is said to be \emph{non-degenerate} if $\lambda_x, \rho_y\in\Sym(X)$, for all $x,y\in X$.
A systematic investigation of set-theoretic solutions started in the late '90s, by Gateva-Ivanova and Van den Bergh \cite{gateva1998semigroups} and Etingof, Schedler, and Soloviev \cite{etingof1998set} who focused on the special class of non-degenerate involutive solutions, i.e., $r^2 = \id_{X\times X}$.  
For convenience, we call involutive non-degenerate solutions simply \textit{solutions} throughout the paper.

A conceptually simple but efficacy strategy for determining all solutions consists of constructing solutions starting from smaller ones. 
To implement this method, it is crucial finding solutions which can not be deconstructed 
into other ones, the so-called indecomposable solutions. 
Precisely, a solution $(X,r)$ is said to be \emph{decomposable} if there exists a partition $\{Y, Z\}$ of $X$ such that $r_{|_{Y\times Y}}$ and $r_{|_{Z\times Z}}$ are still solutions; otherwise, $(X,r)$ is called \textit{indecomposable}. In particular, Etingof, Schedler, and Soloviev in \cite{etingof1998set} proved that every decomposable solution $(X,r)$ can be obtained through two suitable solutions $(Y,s)$ and $(Z,t)$, with $Y$ and $Z$ proper subsets of $X$, $s=r_{|_{Y\times Y}}$, and $t=r_{|_{Z\times Z}}$. Moreover, they showed that the unique indecomposable solution having a prime number $p$ of elements is, up to isomorphism,  the pair $(\mathbb{Z}/p\mathbb{Z}, u)$ where $u(x,y):=(y-1,x+1)$, for all $x,y\in \mathbb{Z}/p\mathbb{Z}$.\\
In recent years, some authors used the links between solutions and algebraic structures to provide new descriptions of indecomposable solutions, as Chouraqui \cite{chouraqui2010garside}, Rump \cite{rump2020}, Smoktunowicz and  Smoktunowicz \cite{smock}, and Pinto together with the first two authors \cite{cacsp2018}.
In particular,  Smoktunowicz and  Smoktunowicz \cite{smock} found a striking connection between these solutions and left braces, a generalization of radical rings, introduced by Rump in \cite{rump2007braces}. As reformulated in \cite{cedo2014braces}, a set $A$ is said to be a \emph{left brace} if it is endowed of two operations $+$ and $\circ$ such that $(A,+)$ is an abelian group, $(A,\circ)$ a group, and 
\begin{align*}
    a\circ (b + c)+ a
    = a\circ b  + a\circ c
\end{align*} 
is satisfied, for all $a,b,c\in A$.
Later, the link between indecomposable solutions and left braces found in \cite{smock} has been developed by Rump \cite{rump2020}. 
Indeed, he showed that a complete description of left braces, on the one hand, and a suitable theory of coverings for solutions, on the other, allows to classify all indecomposable solutions. 
As a first result of this paper, we show that just Theorem $2$ of \cite{rump2020}, together with a result of Smocktunowitcz \cite[Theorem 19]{Sm18} on left braces with nilpotent multiplicative group, reduces the classification of finite indecomposable solutions with nilpotent regular permutation group to the ones having prime-power size.
Alongside these theoretical results, concrete families of indecomposable solutions have been provided by the first two authors together with Pinto \cite{cacsp2018} using Vendramin's dynamical cocycles \cite{vendramin2016extensions}.

Following a different approach, some authors studied indecomposable solutions focusing on the associated \emph{permutation group}, i.e., the group $\mathcal{G}(X,r)$ generated by the set $\{\lambda_x \ | \ x\in X\}$.
Specifically, one of the main objectives is to describe all indecomposable solutions having a specific permutation group. Actually, Pinto, Rump and the first author \cite{capiru2020} developed a method to construct all indecomposable solutions having prime-power size and with cyclic permutation group. In addition, they constructed, distinguishing the isomorphism classes, the indecomposable ones with cardinality $pq$ (where $p$ and $q$ are prime numbers not necessarily distinct) and abelian permutation group. Another remarkable result was recently obtained by Rump, who gave a theoretical classification, via indecomposable cycle sets and cocyclic left braces, all the solutions (not necessarily indecomposable) having cyclic permutation group \cite{Ru20}.
Let us recall that a set $X$ with a binary operation $\cdot$ is called a \emph{cycle set} if each left multiplication $\sigma_x:X\longrightarrow X,\; y\mapsto x\cdot y$ is bijective and
\begin{equation}\label{cicloide}
(x\cdot y)\cdot (x\cdot z)= (y\cdot x)\cdot (y\cdot z) 
\end{equation}
holds, for all $x,y,z\in X$.
Rump introduced this structure in \cite{rump2005decomposition} precisely to investigate solutions. Indeed, he found a one-to-one correspondence between solutions and the class of \emph{non-degenerate} cycle sets.
Up until now, the study of non-degenerate cycle sets has been allowed for obtaining descriptions and constructions of new interesting families of solutions, as one can see in  \cite{rump2016quasi,bon2019,cacsp2017}. 
However, some authors continued the study of indecomposable solutions without using cycle sets. 
Ced\'o, Jespers, and Okni\'nski \cite{CeJeOk20x}, following a suggestion due to Ballester-Bolinches \cite{JeLeRuVe19}, classified indecomposable solutions with primitive permutation group, showing that this class of solutions coincides with the indecomposable ones having a prime number of elements and extending the partial results obtained by Pinto in \cite[Chapter 4]{pinto}. 
In a recent work, Jedli{\v{c}}ka, Pilitowska, and Zamojska-Dzienio \cite{JePiZa20x} concretely classified the indecomposable solutions having abelian permutation group and multipermutation level $2$, extending some results contained in \cite{capiru2020}. These works naturally open a new perspective on the investigations of all indecomposable solutions. Indeed, one can try to describe the indecomposable solutions which have not prime size using further information coming from the imprimitivity blocks. Besides, since all the multipermutation indecomposable solutions provided in \cite{JePiZa20x,capiru2020,cacsp2018,CeJeOk20x} have multipermutation level at most $2$, except \cite[Example 3]{capiru2020}, a first step in this sense can be the concrete construction of indecomposable solutions having greater multipermutation level.

In this spirit, the main aim of this paper is to investigate a class of indecomposable solutions obtained by dynamical cocycles of cycle sets. In particular, we study the dynamical extensions $X\times_{\alpha} S$ for which the complete blocks system $\mathcal{B}_X$ induced by $X$ has an orthogonal system of blocks. For this reason, we call this extensions \textit{orthogonal}.  As one can find in \cite[Definition 2.9]{DoKo09},
given an imprimitive group $G$ acting on a set $X$ of size $mn$, for certain positive integers $n$ and $m$ greater than $1$, we say that two systems of blocks $\mathcal{A}=\{A_1,\ldots,A_m \}$ and $\mathcal{B}=\{B_1,\ldots,B_n \}$ of size $m$ and $n$, respectively, are \textit{orthogonal} if $|A_i\cap B_j|=1$, for all $i\in \{1,\ldots,m\}$ and $j\in \{1,\ldots,n\}$ (for some remarkable results involving orthogonal systems of blocks see, for instance, \cite{Lu91,Do06}). 
In this context, we pay particular attention to the constant dynamical extensions. This class of extensions was recently investigated by a cohomological point of view by Lebed and Vendramin \cite{lebed2017homology}. Following a different approach, using a property of the orthogonal systems of blocks, we show that orthogonal constant dynamical extensions can be constructed by particular semidirect products of cycle sets, a construction of cycle sets introduced by Rump \cite{Ru08}. 
As one can expect, orthogonal dynamical extensions are also related to the semidirect product of left braces. In this regard, the powerful of left braces to check this class of cycle sets becomes clear in the further description of these extensions which we provide.
 
As an application of our results, we give several instances of indecomposable solutions. A lot of these are different from those already contained in \cite{cacsp2018, capiru2020, Ru20,JePiZa20x}. Specifically, we construct indecomposable cycle sets (and so indecomposable solutions) having abelian non-cyclic permutation group and multipermutation level $3$. Moreover, we use the class of \textit{latin} cycle sets, i.e., the cycle sets that also are quasigroups recently studied in \cite{bon2019}, to provide a construction of orthogonal dynamical extensions which give rise to indecomposable irretractable solutions different from the ones obtained in \cite{cacsp2018}.\\
In the last part of this paper, which is self-contained with respect to the previous ones, we consider the class of one-generator left braces, which are closely related to indecomposable solutions (see \cite{smock,rump2020one} for more details). Recently, Rump showed several results involving one-generator left braces: in particular, he described one-generator left braces \cite[Theorem 2]{rump2020one} of multipermutation level at most $2$, giving an answer to \cite[Question 5.5]{smock}. At this end, he introduced a left brace structure over the group ring $\mathbb{Z}[G]$, where $G$ is the infinite cyclic group. Following a different approach, we study one-generator left braces of multipermutation level $2$  
focusing on the map $\lambda$. 
We show that, if $A$ is an arbitrary one-generator left brace of multipermutation level $2$, then the group $\lambda(A)$ is cyclic. As a consequence of this fact, we finally provide a proof of \cite[Theorem 2]{rump2020one} which does not make use of the structure of the left brace over the group ring $\mathbb{Z}[G]$.

\section{Basic results}


In this section, we mainly recall some basics on cycle sets and left braces that are useful throughout the paper. 
In addition, using some well-known results on left braces, we describe the structure  of uniconnected cycle sets with nilpotent permutation group.

\subsection{Solutions of the Yang-Baxter equation and cycle sets}
In \cite{rump2005decomposition}, Rump found a one-to-one correspondence between solutions and a special class of cycle sets, i.e., \emph{non-degenerate} cycle sets.
To illustrate this correspondence, let us firstly recall the following definition.

\begin{defin}[p. 45, \cite{rump2005decomposition}]
A pair $(X,\cdot)$ is said to be a cycle set if each left multiplication $\sigma_x:X\longrightarrow X,$ $y\mapsto x\cdot y$ is invertible and 
$$(x\cdot y)\cdot (x\cdot z)=(y\cdot x)\cdot (y\cdot z), $$
for all $x,y,z\in X$.  Moreover, a cycle set $(X,\cdot)$ is called \textit{non-degenerate} if the squaring map $x\mapsto x\cdot x$ is bijective.
\end{defin}

\begin{conv}
All the cycle sets are non-degenerate throughout the paper.
\end{conv}

\begin{prop}[Propositions 1-2, \cite{rump2005decomposition}]\label{corrisp}
Let $(X,\cdot)$ be a cycle set. Then the pair $(X,r)$, where $r(x,y):=(\sigma_x^{-1}(y),\sigma_x^{-1}(y)\cdot x)$, for all $x,y\in X$, is a solution of the Yang-Baxter equation which we call the associated solution to $(X,\cdot)$. Moreover, this correspondence is one-to-one.
\end{prop}

Recall that a first useful tool to construct new solutions, introduced in \cite{etingof1998set}, is the so-called \textit{retract relation}, an equivalence relation on $X$ which we denote by $\sim_r$. Precisely, if $(X,r)$ is a solution, then $x\sim_r y$ if and only if $\lambda_x=\lambda_y$, for all $x,y\in X$. In this way, $(X,r)$ induces another solution, having the quotient $X/\sim_r$ as underlying set, which is named \textit{retraction} of $(X,r)$ and is denoted by $\Ret(X, r)$. 
As one can expect, the retraction of a solution corresponds to the retraction of a non-degenerate cycle set. Specifically, in \cite{rump2005decomposition} Rump showed that the binary relation $\sim_\sigma$ on $X$ given by 
$$x\sim_\sigma y :\Longleftrightarrow \sigma_x = \sigma_y$$ 
for all $x,y\in X$, is a congruence of $(X,\cdot)$ and he proved that the quotient $X/\sim$, which we denote by $\Ret(X)$, is a cycle set whenever $X$ is non-degenerate and he called it the \emph{retraction} of $(X,\cdot)$. As the name suggests, if $(X, \cdot)$ is the cycle set associated to a solution $(X,r)$, then the retraction $\Ret(X)$ is the cycle set associated to $\Ret(X,r)$. Besides, a cycle set $X$ is said to be \textit{irretractable} if $\Ret(X)=X$, otherwise it is called \textit{retractable}.

\smallskip

The following definition is of crucial importance for our scopes.

\begin{defin}
    A cycle set $X$ has \textit{multipermutation level $n$} if $n$ is the minimal non-negative integer such that $\Ret^n(X)$ has cardinality one, where 
    $$\Ret^0(X):=X \text{ \ and \ } \Ret^i(X):=\Ret(\Ret^{i-1}(X)) \text{, \ for }i>0. $$  
\end{defin}
\medskip

\noindent Clearly, a cycle set of multipermutation level $n$ is retractable, but the converse is not necessarily true.\\

\noindent For a cycle set $X$, the permutation group generated by the set $\{\sigma_x \, | \, x\in X\}$ will be denoted by $\mathcal{G}(X)$ and we call it the \textit{associated permutation group}. Obviously, in terms of solutions, the associated permutation group is exactly the permutation group generated by the set $\{\lambda_x \, | \, x\in X\}$.
\medskip

\noindent Our attention is mainly posed on cycle sets that are indecomposable. 
\begin{defin}
A cycle set $(X,\cdot)$ is said to be \textit{indecomposable} if the permutation group $\mathcal{G}(X)$ acts transitively on $X$. Moreover, $X$ is said to be \textit{uniconnected} if $\mathcal{G}(X)$ acts regularly on $X$, i.e., $\mathcal{G}(X)$ acts freely and transitively on $X$.
\end{defin}

\noindent Note that a solution $(X,r)$ is indecomposable if and only if the associated cycle set $(X,\cdot)$ is indecomposable. In the rest of the paper, we will study indecomposable solutions by their associated cycle sets. In every case, all the results involving cycle sets can be translated in terms of solutions by \cref{corrisp}.

\medskip

\noindent A particular family of indecomposable cycle sets that has been considered by some authors (see, for example, \cite{etingof1998set,capiru2020}) is that of cycle sets having cyclic permutation group. In \cite{capiru2020}, a method to construct all the indecomposable cycle sets of prime-power size with cyclic permutation group and multipermutation level greater than $1$ was developed.

\begin{theor}[Theorems 8 - 9, \cite{capiru2020}]\label{costruz2}
Let $p$ be a prime number, $X:=\{0,\dots,p^k-1\}$, $n\in \mathbb{N}\setminus\{1\}$, $j_0,\dots,j_{n}\in \mathbb{N}\cup\{0\}$ such that $j_{n}=0$, $j_0=k$, $j_i<j_{i-1}$, for every $i\in \{1,\dots,n\}$, and $\{f_i\}_{i\in \{1,\dots,n-1\}}$ a set of maps such that
$$f_{i}:\mathbb{Z}/p^{j_i}\mathbb{Z}\longrightarrow \{0,\dots,p^{j_{i-1}}/p^{j_i}-1\}$$
$f_{i}(0)=0$, for every $i\in \{1,\dots,n-1 \}$, and the map
$$\varphi_{i}:\{0,\dots,p^{j_{i}}-1\}\longrightarrow \{1,\dots,p^{j_{i-1}}-1\}$$
$$l\mapsto 1+p^{j_{n-1}}f_{n-1}(l)+\cdots+p^{j_{i}}f_{i}(l)$$
is injective, for every $i\in \{1,\dots,n-1\}$. Moreover, set $\psi:=(0\;\dots\;p^k-1)$ and
 \begin{equation}
\sigma_{i}:=\psi^{1+p^{j_{n-1}}f_{n-1}(i)+\cdots+p^{j_1}f_{1}(i)},
\end{equation}
for every $i\in X$. Define 
$$K_{j,i}:=j+1+p^{n-1}f_{n-1}(i)+\cdots+p^{j_2}f_2(i)$$
and
$$Q_{j,i}:= p^{j_{n-1}}f_{n-1}(i)+\cdots+p^{j_1}f_1(i)+p^{j_{n-1}}f_{n-1}(K_{j,i})+\cdots+p^{j_1}f_1(K_{j,i}),$$
for every $i,j\in X$, and suppose that $Q_{i,j}\equiv Q_{j,i}\;(mod\;p^k)$ for every $i,j\in \{0,\ldots,p^k-1\}$.\\
Then $X$ is an indecomposable cycle set of level $n$ and cyclic permutation group $<\psi>$ such that $|\Ret^i(X)|=p^{j_i}$, for every $i\in\{0,\ldots,n\}$.\\
Conversely, every indecomposable cycle set with cyclic permutation group arises in this way.
\end{theor}

\subsection{Solutions of the Yang-Baxter equation and left braces}
At first, we introduce the following definition that, as observed in \cite{cedo2014braces}, is equivalent to the original introduced by Rump in \cite{rump2007braces}.

\begin{defin}[\cite{cedo2014braces}, Definition 1]
A set $A$ endowed of two operations $+$ and $\circ$ is said to be a \textit{left brace} if $(A,+)$ is an abelian group, $(A,\circ)$ a group, and
\begin{align*}
    a\circ (b + c) + a
    = a\circ b + a\circ c,
\end{align*}
for all $a,b,c\in A$.
\end{defin}

Given a left brace $A$ and $a\in A$, let us denote by $\lambda_a:A\longrightarrow A$ the map from $A$ into itself defined by
\begin{align}\label{eq:gamma}
    \lambda_a(b):= - a + a\circ b,
\end{align} 
for all $b\in A$. 
As shown in \cite[Proposition 2]{rump2007braces} and \cite[Lemma 1]{cedo2014braces}, these maps have special properties. We recall them in the following proposition.
\begin{prop}\label{action}
Let $A$ be a left brace. Then, the following are satisfied: 
\begin{itemize}
\item[1)] $\lambda_a\in\Aut(A,+)$, for every $a\in A$;
\item[2)] the map $\lambda:A\longrightarrow \Aut(A,+)$, $a\mapsto \lambda_a$ is a group homomorphism from $(A,\circ)$ into $\Aut(A,+)$.
\end{itemize}
\end{prop}

\noindent The map $\lambda$ is of crucial importance to construct solutions of the Yang-Baxter equation using left braces, as one can see in the following proposition.

\begin{prop}[Lemma 2, \cite{cedo2014braces}]
Let A be a left brace and $r:A\times A\rightarrow A\times A$ the map given by
$$
r(a,b):=(\lambda_a(b),\lambda^{-1}_{\lambda_a(b)}(a)), $$
for all $a,b\in A$. Then, $(A,r)$ is a solution of the Yang-Baxter equation which we call the associated solution to the left brace $A$.
\end{prop}

For the following definition, we refer the reader to \cite[p. 160]{rump2007braces} and \cite[Definition 3]{cedo2014braces}.
\begin{defin}
Let $A$ be a left brace. A subset $I$ of $A$ is said to be a \textit{left ideal} if it is a subgroup of the multiplicative group and $\lambda_a(I)\subseteq I$, for every $a\in A$. Moreover, a left ideal is an \textit{ideal} if it is a normal subgroup of the multiplicative group.
\end{defin}
\noindent Given an ideal $I$, it holds that the structure $A/I$ is a left brace called the \emph{quotient left brace} of $A$ modulo $I$.
%
\medskip

More in general, examples of left ideals can be found easily in a finite left brace, as the following proposition shows (see \cite[p. 11]{bachiller2015extensions}).
\begin{prop}\label{rideal}
Let $A$ be a finite left brace and $p$ a prime number dividing $|A|$. Then, the $p$-Sylow subgroup $A_p$ of the additive group $(A,+)$ is a left ideal of $A$.
\end{prop}
\medskip

In \cite{rump2007braces}, Rump introduced the special notion of the socle of a left brace that, in the terms of \cite[Section 4]{cedo2014braces}, is the following.
\begin{defin}
Let $A$ be a left brace. Then, the set 
\begin{align*}
    \Soc(A) := \{a\in A \ | \ \forall \,
     b\in A \quad  a + b = a\circ b \}
\end{align*}
is named \emph{socle} of $A$.
\end{defin}
\noindent Clearly,
$\Soc(A) := \{a\in A \ | \ \lambda_a = \id_A\}$.
Moreover, we have that $\Soc(A)$ is an ideal of $A$.
In \cite{rump2007braces}, Rump defined the \emph{socle series of $A$} by setting
$\Soc_0(A):= \{0\}$ and 
\begin{align*}
  \Soc_{n}(A):=  \{a \ | \  \forall \, b\in A \quad a + b - a\circ b\in \Soc_{n-1}(A) \},
\end{align*}
for every $n\in\mathbb{N}$.
Obviously, $\Soc_1(A)$ coincides with $\Soc(A)$.
\bigskip

As highlighted in \cite[Proposition 7]{rump2007braces} and \cite[Lemma 3]{cedo2014braces}, the socle of a left brace $A$ has a special linkage with the solution associated to $A$. 
\begin{prop}
Let $A$ be a left brace and $(A,r)$ the solution associated to $A$. If $(A/ \Soc(A), r)$ is the solution associated to the left brace $A/\Soc(A)$, then $(A/\Soc(A), r)$ coincides with the retraction $\Ret(A,r)$ of $(A,r)$.
\end{prop}
\smallskip

\begin{defin}
A left brace $A$ has \emph{finite multipermutation level} if the sequence $S_n$ defined as $S_1 := A$
and $S_{n + 1} = S_n/\Soc(S_n)$ for $n\geq 1$ reaches zero.
\end{defin}

\begin{prop}
Let $A$ be a left brace. Then, $A$ has finite multipermutation level if and only if $A$ admits an $s$-series, i.e., there exists a positive integer $n$ such that $A = \Soc_n(A)$.
\end{prop}
In particular, if $n$ is the smallest positive integer such that $A = \Soc_n(A)$, then $A$ is said to be a \emph{left brace of multipermutation level $n$}.

\subsection{Left braces and uniconnected cycle sets with nilpotent permutation group}

In this subsection, we use some well-known results involving left braces to determine the structure of uniconnected cycle sets with nilpotent permutation group. As a consequence, 
we use this result and the ones contained in \cite{capiru2020}
to provide a method for constructing all the indecomposable cycle sets with cyclic permutation group. 
\medskip

Below we recall two results, the first due to Smocktunowitcz \cite[Theorem 19]{Sm18} and the second to Rump \cite{rump2020}, which are of crucial importance for our purposes.

\begin{lemma}[Theorem 19, \cite{Sm18}]\label{brace}
Let $A$ be a finite left brace such that $(A,\circ)$ is nilpotent and suppose that $|A|=p_1^{n_1}\cdots p_m^{n_m}$. Then, $A$ and $\prod\limits_{i}A_{p_i}$ are isomorphic as left braces.
\end{lemma}
 
Before introducing Rump's Theorem on uniconnected cycle sets, recall that if $A$ is a left brace, by \cref{action}, the maps $\lambda_a$ in \eqref{eq:gamma} 
determines an action of $(A,\circ)$ on $(A,+)$. According to 
\cite{rump2007braces, rump2020}, a subset $X$ of $A$ which is a union of orbits with respect to such an action
and generating the additive group $(A,+)$ is called \textit{cycle base}. Moreover, if a cycle base is a single orbit then is said to be a \textit{transitive cycle base}. 
\begin{lemma}[Theorem 2, \cite{rump2020}]\label{rump20}
Let A be a left brace, $X$ a transitive cycle base and $g\in X$. Define on $A$ the binary operation $\bullet$ 
$$
a\bullet b:=(\lambda_a(g))^{-}\circ b,
$$
for all $a,b\in A$. Then, $(A,\bullet)$ is a uniconnected cycle set and $\mathcal{G}(A)\cong (A,\circ)$. Conversely, every uniconnected cycle set  can be constructed in this way.
\end{lemma}

If $A$ is a left brace and $(A,\bullet)$ is the uniconnected cycle set constructed as in the previous lemma, from now on we refer to $(A,\bullet)$ as the cycle set \textit{associated} to the left brace $A$. Now, we provide the main result of this section, namely we show that the classification of finite indecomposable cycle sets with a nilpotent regular permutation group reduces to the classification of the ones having prime-power size.

\begin{theor}\label{princ}
Let $m,n_1,...,n_m$ be natural numbers, $p_1,...,p_m$ distinct prime numbers and $(X,\cdot)$ an uniconnected cycle set of cardinality $p_1^{n_1}\cdots p_m^{n_m}$ and with nilpotent permutation group $\mathcal{G}(X)$. Then, there exist $m$ indecomposable cycle sets $(Y_1,\cdot_1),\ldots,(Y_m,\cdot_m)$ such that $|Y_i|=p_i^{n_i}$ and 
$$(X,\cdot)\cong (Y_1,\cdot_1)\times\cdots\times (Y_m,\cdot_m).$$
\end{theor}

\begin{proof}
By \cref{rump20}, there exists a left brace $(A,+,\circ)$ such that the associated cycle set $(A,\bullet)$ is isomorphic to $(X,\cdot)$ and $\mathcal{G}(X)\cong (A,\circ)$.\\
Now, since $(A,\circ)$ is nilpotent, by \cref{brace} we have that $(A,+,\circ)\cong \prod\limits_{i} (A_{p_i},+,\circ)$, where $A_{p_i}$ is the $p_i$-Sylow of $(A,\circ)$. Hence, by a routine computation, it follows that $(A,\bullet)$ is isomorphic to the direct product $\prod\limits_{i} (A_{p_i},\bullet)$, where $(A_{p_i},\bullet)$ is the uniconnected cycle set associated to $(A_{p_i},+,\circ)$. Setting $(Y_i,\cdot_i):=(A_{p_i},\bullet)$, for every $i\in \{1,\ldots,m\}$, the claim follows.
\end{proof}

As a first consequence of the previous theorem, we obtain an alternative proof of \cite[Proposition 4]{Ru20}.

\begin{cor}[Proposition 4, \cite{Ru20}]
Let $(X,\cdot)$ be an indecomposable cycle sets having cyclic permutation group $\mathcal{G}(X)$. Then, $<\sigma_x>=\mathcal{G}(X)$, for every $x\in X$.
\end{cor}

\begin{proof}
If $X$ has prime-power size, the claim follows by \cite[Lemma 4-5]{capiru2020}, otherwise, the claim follows by \cite[Lemma 4-5]{capiru2020}  and \cref{princ}.
\end{proof}

In the following corollary, we specialize our result to indecomposable cycle sets $X$ with cyclic permutation group $\mathcal{G}(X)$. 
Observe that, by this hypothesis together with that of transitivity on $X$, we have that $\mathcal{G}(X)$ is a nilpotent group that acts regularly on $X$, i.e.,  $X$ is uniconnected.

\begin{cor}
Let $m,n_1,...,n_m$ be natural numbers and $p_1,...,p_m$ distinct prime numbers. Suppose that $(X,\cdot)$ is an indecomposable cycle set of cardinality $p_1^{n_1}\cdots p_m^{n_m}$ and with cyclic permutation group $\mathcal{G}(X)$. Then, there exist $m$ indecomposable cycle sets $(Y_1,\cdot_1),\ldots,(Y_m,\cdot_m)$ such that $|Y_i|=p_i^{n_i}$ and 
$$(X,\cdot)\cong (Y_1,\cdot_1)\times\cdots\times (Y_m,\cdot_m).$$
Moreover, $(Y_i,\cdot_i)$ is a cycle set of multipermutation level $1$ or a cycle set constructed as in \cref{costruz2}.
\end{cor}
\begin{proof}
Since $X$ is uniconnected, by the previous theorem, there exist $m$ indecomposable cycle sets $(Y_1,\cdot_1),\ldots,(Y_m,\cdot_m)$ such that $|Y_i|=p_i^{n_i}$ and 
$$(X,\cdot)\cong (Y_1,\cdot_1)\times\cdots\times (Y_m,\cdot_m).$$
Since $\mathcal{G}(X)$ is cyclic so $\mathcal{G}(Y_i)$ is, hence $(Y_i,\cdot_i)$ is a cycle set constructed as in \cref{costruz2}.
\end{proof}
\medskip

We highlight that to classify concretely finite indecomposable cycle sets with nilpotent regular permutation group,  by \cref{princ}, partial results have been obtained by \cref{costruz2}. 
However, these results do not consider all cycle sets belonging to this family, as one can see in the following example.

\begin{ex}
Let $X:=\{1,2,3,4,5,6,7,8 \}$ be the cycle set given by
$$\sigma_1=\sigma_6:=(1\;4)(2\;8)(3\;7)(5\;6)\qquad \sigma_2=\sigma_7:=(1\;7\;6\;2)(3\;4\;8\;5) $$
$$\sigma_3=\sigma_8:= (1\;5)(2\;3)(4\;6)(7\;8)  \qquad \sigma_4=\sigma_5:=(1\;2\;6\;7)(3\;5\;8\;4).$$
Then, $\mathcal{G}(X)$ is isomorphic to the dihedral group of order $8$ and acts transitively on $X$.
\end{ex}



We conclude this section by showing that the classification of left braces has a great impact on indecomposable cycle sets (and hence on indecomposable solutions) also in the infinite case. To this end, we recall the following result due to Ced\'o, Smocktunowitcz, and Vendramin.

\begin{lemma}[\cite{CeSmVe19}, Theorem 5.5]
Let A be a left brace such that $(A,\circ)\cong (\mathbb{Z},+)$. 
Then, $a + b = a\circ b$, for all $a,b\in A$, i.e., $A$ is a trivial left brace.
\end{lemma}
\smallskip

As a consequence of the previous lemma we provide the following significant theorem. One can note that it is essentially contained in the recent preprint by Jedli{\v{c}}ka, Pilitowska, and Zamojska-Dzienio \cite[Proposition 6.2]{JePiZa20x}, however here we give an alternative proof of their result.
\begin{theor}
Let $(X,\cdot)$ be an infinite indecomposable cycle set such that $\mathcal{G}(X)\cong (\mathbb{Z},+)$. Then, up to isomorphism, $(X,\cdot)$ is the cycle set on $\mathbb{Z}$ given by $a\cdot b:=b+1$, for all $a,b\in \mathbb{Z}$.
\end{theor}

\begin{proof}
By \cref{rump20}, there exists a left brace $A$ such that the associated  cycle set $(A,\bullet)$ is isomorphic to $(X,\cdot)$ and $\mathcal{G}(A)\cong (A,\circ)$, hence $(A,\circ)\cong(\mathbb{Z},+)$. 
By the previous lemma, $A$ is the trivial left brace and, consequently, $\lambda_a = \id_A$, for every $a\in A$. 
In this way, the unique transitive cycle bases of $A$ are $\{1\}$ and $\{-1\}$. By \cref{rump20}, since $a\bullet b = (\lambda_a(g))^{-}\circ b$, for some element $g$ of a transitive cycle base, we obtain that $a\bullet b = b + 1$, for all $a,b\in A$, if $g = -1$, and $a\bullet b = b - 1$, for all $a,b\in A$, if $g=1$.
Therefore, the claim follows.
\end{proof}

\section{Orthogonal dynamical extensions of cycle sets}

In this section, we consider a particular type of dynamical extensions which we call orthogonal. Specifically, we focus on those that are constant, and we show that, for these extensions, one can simplify the cocycle-condition \eqref{cociclo}. This result allows for adapting \cite[Theorem 7]{cacsp2018} to characterize them.
\medskip

Let us begin by recalling some suitable notions (see for instance \cite{dixon1996permutation}). If $G$ is a finite transitive group acting on a set $X$, then $G$ is said to be \textit{imprimitive} if there exists a subset $B$ of $X$, $B\neq X$, with at least two elements such that, for each permutation $g$ of $G$, either $g(B)=B$ or $g(B)\cap B=\emptyset$. The subsets $g(B)$ of $X$ are called \textit{blocks} and the set $\{g(B)\}_{g\in G}$ is said to be a \textit{system of blocks}. In this way, the action of $G$ on $X$ induces an action on a system of blocks in a natural way. \\
Let $G$ be an imprimitive group acting on a set $X$ of size $mn$, for some natural numbers $n$ and $m$ greater than $1$, and let $\mathcal{A}=\{A_1,\ldots,A_m \}$ and $\mathcal{B}=\{B_1,\ldots,B_n \}$ be two systems of blocks of size $m$ and $n$. Referring to \cite{DoKo09}, we say that two systems of blocks $\mathcal{A}$ and $\mathcal{B}$ are \textit{orthogonal} if $|A_i\cap B_j|=1$, for all $i\in \{1,\ldots,m\}$ and $j\in \{1,\ldots,n\}$.
The following classical lemma involving imprimitive groups is useful for our purpose.

\begin{lemma}[Lemma 2.2, \cite{Do06}]\label{quagr}
Let $G$ be an imprimitive group acting on a set $X$ and let $\mathcal{A}$ and $\mathcal{B}$ be two orthogonal systems of blocks. Then, the action of $G$ on $X$ is equivalent to the action of $G$ on $\mathcal{A}\times \mathcal{B}$ given by $g(A,B):=(g(A),g(B))$, for all $A\in \mathcal{A}$, $B\in \mathcal{B}$.
\end{lemma}

Following the paper by Vendramin \cite{vendramin2016extensions}, if $X$ is a cycle set, $S$ a set and $\alpha:X\times X\times S\longrightarrow \Sym(S)$, $\alpha(x,y,s)\mapsto\alpha_{(x,y)}(s,-)$ a function such that
\begin{equation}\label{cociclo}
\alpha_{(x\cdot y,x\cdot z)}(\alpha_{(x,y)}(r,s),\alpha_{(x,z)}(r,t))=\alpha_{(y\cdot x,y\cdot z)}(\alpha_{(y,x)}(s,r),\alpha_{(y,z)}(s,t)),
\end{equation}
for all $x,y,z\in X$ and $r,s,t\in S$, then $\alpha$ is said to be a \textit{dynamical cocycle} and the operation $\cdot$ given by 
$$
(x,s)\cdot (y,t):=(x\cdot y,\alpha_{(x,y)}(s,t)),
$$
for all $x,y\in X$ and $s,t\in S$ makes $X\times S$ into a cycle set which we denote by $X\times_{\alpha} S$ and we call \textit{dynamical extension} of $X$ by $\alpha$. A dynamical cocycle is said to be \textit{constant} if $\alpha_{(x,y)}(s,-)=\alpha_{(x,y)}(t,-)$, for all $x,y\in X$ and $s,t\in S$. In this case, $X\times_{\alpha} S$ is said to be a \textit{constant dynamical extension}.\\
Moreover, a dynamical extension $X\times_{\alpha} S$ is called \textit{indecomposable} if $X\times_{\alpha} S$ is an indecomposable cycle set: by \cite[Theorem 7]{cacsp2018}, this happens if and only if $X$ is an indecomposable cycle set and the subgroup of $\mathcal{G}(X\times S)$ generated by $\{h \ | \ \forall \, s\in S \quad h(y,s)\in \{y\}\times S \}$ acts transitively on $\{y\}\times S$, for some $y\in X$. In this case, $\mathcal{G}(X\times S)$ acts imprimitively and the set $\{ \{x\}\times S\}_{x\in X}$ is a system of blocks which we call the \textit{system of blocks induced by $X$} and we denote it by $\mathcal{B}_{X}$.

\begin{defin}
Given a finite indecomposable cycle set $X$ and $X\times_{\alpha} S$ a (constant) indecomposable dynamical extension of $X$ by $\alpha$, we say that $X\times_{\alpha} S$ is
a \emph{(constant) orthogonal dynamical extension of $X$ by $\alpha$} if the permutation group $\mathcal{G}(X\times S)$ has a system of blocks $\mathcal{A}$ which is orthogonal to $\mathcal{B}_X$. 
\end{defin}

\medskip

Let us observe that one can find examples of orthogonal dynamical extensions in \cite{cacsp2018}, \cite{capiru2020}, and \cite{JePiZa20x}. Below, we recall some of them.

\begin{exs}
\hspace{1mm}
\begin{enumerate}
    \item[1)] If $X$ and $Y$ are indecomposable cycle sets of multipermutation level $1$ and having coprime cardinalities, the direct product $X\times Y$ coincides with the orthogonal dynamical extension of $X$ by $\alpha$, where $\alpha:X\times X\times Y\rightarrow \Sym(Y)$ is the function given by $\alpha_{(x,y)}(s,-):=\sigma_s$, for all $x,y\in X$ and $s\in Y$.
    \item[2)] Let $p$ be a prime number and $X=Y:=\mathbb{Z}/p\mathbb{Z}$ and consider $X$ as the cycle set given by $x\cdot y:=y+1$, for all $x,y\in X$. Moreover, let $\alpha$ be the function from $X\times X\times Y$ to $\Sym(Y)$ given by $\alpha_{(x,y)}(s,t):=t+x$, for  all $(x,y),(s,t)\in X\times Y$. Then $X\times_{\alpha} Y$ is an orthogonal dynamical extension of $X$ by $\alpha$.
\end{enumerate}
\end{exs}

However, several examples of dynamical extension that are not orthogonal occur in literature.

\begin{exs}
\hspace{1mm}
\begin{enumerate}
    \item[1)] Every indecomposable dynamical extension having prime-power size and cyclic permutation group can not be obtained as an orthogonal dynamical extension. 
    \item[2)] Let $k$ be a natural number,  $X:=\mathbb{Z}/k\mathbb{Z}$ the cycle set given by $x\cdot y:=y+1$, for all $x,y\in X$, $A = B:=\mathbb{Z}/2\mathbb{Z}$, and $S:=A\times B$. Moreover, define $\beta:A\times A\times X\rightarrow \Sym(B)$ the function given by $\beta_{(a,b)}(x,-):=\id_B$, if $a=b$, and $\beta_{(a,b)}(x,-):=(0\;1)$, otherwise, define $\gamma:B\rightarrow \Sym(A)$  the function given by $\gamma_b(a):=a-b-1$, for  
    every $b\in B$, and let $\alpha:X\times X\times S\rightarrow \Sym(S)$ be the function given by
    \begin{center}
$\alpha_{(x,y)}((a,b),(c,d)):=\begin{cases} (c,\beta_{(a,c)}(x,d)), & \mbox{if }x= y \\ (\gamma_b(c),d), & \mbox{if }\mbox{ }x\neq y\end{cases},$
\end{center}
for all $x,y\in X$, $(a,b),(c,d)\in S$. Then, $X\times_{\alpha} S$ is a dynamical extension of $X$ by $\alpha$ that is not orthogonal. Indeed, if $\mathcal{A}=\{A_1,\dots,A_k \}$ is a system of blocks orthogonal to $\mathcal{B}_X$, then every block $A_i$ is of the form $A_i:=\{(1,a_1,b_1),\ldots,(k,a_k,b_k)\}$. But one can show that there 
exists $g\in \mathcal{G}(X\times_{\alpha} S)$ such that $g(1,a_1,b_1)=(1,a_1,b_1)$ and $g(k,a_k,b_k)= (k,a_g,b_k) $ for an element $a_g$ different from $a_k$, a contradiction.
\end{enumerate}
\end{exs}

\medskip

From now on, in the rest of this section we focus on constant orthogonal dynamical extensions. As noted in \cite[Section 3]{vendramin2016extensions}, to obtain constant dynamical extensions, the cocycle-condition \eqref{cociclo} can be simplified. Now, we show that for constant orthogonal dynamical extensions the condition \eqref{cociclo} can be even more simplified.


\begin{prop}\label{opera}
Let $X\times_{\alpha} S$ be a constant orthogonal dynamical extension of $X$ by $\alpha$. Then, after renaming the elements of the set $S$, the function $\alpha$ depends only on the first variable, i.e., $\alpha_{(x,y)}(-,t)=\alpha_{(x,y')}(-,t)$, for all $x,y,y'\in X$ and $t\in S$.
\end{prop}

\begin{proof}
Let $q$ be a natural number such that $\mathcal{A}:=\{A_1,...,A_q\}$ is a system of blocks  orthogonal to $\mathcal{B}_X$. 
Thus, we necessarily have that 
$$A_i=\{ (x_{i_1},s_{i_1}),\ldots, (x_{i_{|X|}},s_{i_{|X|}})\}$$
where $x_{i_k}\neq x_{i_j}$ if $k\neq j$, for every $i\in \{1,...,q\}$. Therefore, without loss of generality, after renaming the elements of $S$, we can suppose that $s_{i_1}=\cdots=s_{i_{{|X|}}}$, for every $i\in \{1,...,q\}$. In this way, for every $A\in \mathcal{A}$, there exists an element $s\in S$ such that $A=X\times \{s\}$, hence we can identify $\mathcal{B}_X\times \mathcal{A}$ with $X\times S$. By \cref{quagr}, we have 
$$
(x,s)\cdot (y,t)=(\delta_{(x,s)}(y),\delta_{(x,s)}(t) ),
$$
for all $x,y\in X$ and $s,t\in S$, where $\delta_{(x,s)} $ is an element of $\mathcal{G}(X\times_{\alpha} S)$ depending on $x$ and $s$. It follows that  
$$
(x\cdot y,\alpha_{(x,y)}(-,t))=(\delta_{(x,s)}(y),\delta_{(x,s)}(t) ),
$$
for all $x,y\in X$ and $s,t\in S$, therefore $\alpha_{(x,y)}(-,t)=\delta_{(x,s)}(t) $, i.e., $\alpha$ only depends on $x$.
\end{proof}

Therefore, a constant orthogonal dynamical extension $X\times_{\alpha} S$ can be thought as a cycle set on $X\times S$ obtained by a map $\beta:X\longrightarrow \Sym(S),$ $x\mapsto \beta_x$ from $X$ to the symmetric group of $S$ such that $\beta_{x\cdot y}\beta_x=\beta_{y\cdot x}\beta_y$ and $\alpha_{(x,y)}(s,-)=\beta_x$, for all $x,y\in X$, $s\in S$.
\smallskip

As a consequence of the previous proposition, we give a characterization for constant orthogonal dynamical extensions $X\times_{\alpha} S$.

\begin{theor}\label{carindec}
Let $X$ be a finite indecomposable cycle set. Moreover, let $S$ be a finite set and $\alpha:X\longrightarrow \Sym(S)$ a map from $X$ to the symmetric group of $S$ such that $\alpha_{x\cdot y}\alpha_x=\alpha_{y\cdot x}\alpha_y$, for all $x,y\in X$. Then, the binary operation on $X\times S$ given by 
$$
(x,s)\cdot (y,t):=(x\cdot y,\alpha_{x}(t)),
$$
makes $X\times S$ into a constant orthogonal dynamical extension of $X$ by $\alpha$ if and only if the subgroup $H_y$ of $\Sym(X\times S)$ given by
$$
H_y:=\{ (\sigma^{\epsilon_1}_{x_1}\cdots\sigma^{\epsilon_n}_{x_n},\alpha^{\epsilon_1}_{x_1}\cdots\alpha^{\epsilon_n}_{x_n})\ | \  n\in \mathbb{N},\ \epsilon_1,...,\epsilon_n\in \{-1,1\} \quad \sigma^{\epsilon_1}_{x_1}\cdots\sigma^{\epsilon_n}_{x_n}(y)=y\} 
$$
acts transitively on $\{y\}\times S$, for some $y\in X$. In this case, the set $\{X\times \{s\}\}_{s\in S}$ is a system of blocks orthogonal to $\mathcal{B}_X$.\\
Moreover, every constant orthogonal dynamical extension can be constructed in this way.
\end{theor}

\begin{proof}
By a long but simple calculation, one can verify that $(X\times S,\cdot)$ is a dynamical extension of $X$ by $\alpha$. Moreover, by \cite[Theorem 7]{cacsp2018} we have that it is indecomposable if and only if $H_y$ acts transitively on $\{y\}\times S$, for some $y\in X$. Now, since
$$
\sigma_{(x,s)}(y,t)\in X\times \{\alpha_x(t)\},
$$
for all $x,y\in X,$ $s,t\in S$, it follows that $\{X\times \{s\}\}_{s\in S}$ is a system of blocks. Obviously, it is orthogonal to $\mathcal{B}_X$.\\
The rest of the  
assertion follows by \cref{opera} and \cite[Theorem 7]{cacsp2018}.
\end{proof}
\medskip
\begin{rem}\label{rema}
Given two cycle sets $X$ and $S$ and $\alpha:X\rightarrow \Aut(S)$ a function such that $\alpha_{x\cdot y}\alpha_x=\alpha_{y\cdot x}\alpha_y$, for all $x,y\in X$, then the binary operation $\cdot$ on $X\times S$ given by $(x,s)\cdot (y,t):=(x\cdot y,\alpha_x(s\cdot t))$ makes $X\times S$ into a cycle set. 
Moreover, it is easy to show that it is isomorphic to the Rump's semidirect product of the cycle sets $X$ and $S$ via $\alpha$, a construction of cycle sets introduced by Rump in \cite{Ru08}. Since a set $S$ can be endowed with a cycle set structure by $x\cdot y:=y$, for all $x,y\in S$, the previous results ensure that every constant orthogonal dynamical extension can be obtained as particular semidirect product of cycle sets. On the other hand, one can show that every semidirect product $X\ltimes_{\alpha} S$ of two finite cycle sets $X$ and $S$ via $\alpha$ gives rise to orthogonal dynamical extensions whenever  $X\ltimes_{\alpha} S$ is indecomposable. 
\end{rem}

\cref{carindec} turns out to be useful in the investigation of indecomposable cycle sets for which the retraction has an orthogonal system of blocks. Indeed, if $X$ is a finite indecomposable cycle set, by \cite[Proposition 8]{cacsp2018}, $X$ is isomorphic to a constant dynamical extension of $\Ret(X)$, hence we have the following corollary.

\begin{cor}
Let $X$ be a finite indecomposable cycle set such that $\mathcal{B}_{\Ret(X)}$ has an orthogonal system of blocks. Then, there exist a set $S$ and a function $\alpha:\Ret(X)\longrightarrow \Sym(S)$ such that the operation $\cdot$ given by
$$
(x,s)\cdot (y,t):=(x\cdot y,\alpha_{x}(t)),
$$
for all $x,y\in \Ret(X)$ and $s,t\in S$ makes $\Ret(X)\times S$ into a cycle set isomorphic to $I$.
\end{cor}

\begin{proof}
It follows by \cite[Proposition 8]{cacsp2018} and \cref{carindec}.
\end{proof}

In other words, the previous corollary states that to construct all the finite indecomposable cycle sets having $X$ as retraction and such that there exists a system of blocks orthogonal to $\mathcal{B}_{X}$ one has to classify all the pair $(S,\alpha)$ where $S$ is a set, $\alpha$ is a function from $X$ to $\Sym(S)$ such that $\alpha_{x\cdot y}\alpha_x=\alpha_{y\cdot x}\alpha_y$, for all $x,y\in X$, and the subgroup $H_y$ of  \cref{carindec} acts transitively on $\{y\}\times S$. 
\bigskip

Note that if $X\times_{\alpha} S$ is a constant indecomposable dynamical extension, the system of blocks $\mathcal{B}_{\Ret(X\times_{\alpha} S)}$ induced by the retraction is in general different from the system of blocks $\mathcal{B}_{X}$ induced by $X$ (see, for instance, \cref{refret}). However, $\mathcal{B}_{\Ret(X\times_{\alpha} S)}$ and $\mathcal{B}_{X}$ are closely related, as one can see in the simple but useful \cref{refinement}. Before introducing such a proposition, let us recall that if $G$ is an imprimitive group acting on a finite set $X$ and $\mathcal{A},\mathcal{B}$ are system of blocks, then $\mathcal{A}$ is said to be a \textit{refinement} of $\mathcal{B}$ if there exist $A\in \mathcal{A}$, $B\in \mathcal{B}$ such that $A\subseteq B$.

\begin{prop}\label{refinement}
Let $X\times_{\alpha} S$ be a constant indecomposable dynamical extension. Then, $\mathcal{B}_{X}$ is a refinement of $\mathcal{B}_{\Ret(X\times_{\alpha} S)}$.
\end{prop}

\begin{proof}
By \cite[Propositon 2]{cacsp2018}, we have that 
$$
(x,s)\cdot (y,t):=(x\cdot y,\alpha_{x,y}(t)),
$$
for all $x,y\in X$ and $s,t\in S$. Therefore, if $B\in \mathcal{B}_{\Ret(X\times_{\alpha} S)}$ and $(x,s)\in B$, then $\{x\}\times S \subseteq B$, hence the claim follows.
\end{proof}

\section{Description of constant orthogonal dynamical extensions by left braces}
 
The goal of this section is to give a 
relationship between constant orthogonal dynamical extensions and semidirect product of left braces. This link naturally leads to a description of constant orthogonal dynamical extensions and provides a way to investigate these extensions by left  braces.
\medskip

For our purpose, it is useful to recall the semidirect product of left braces  in the same terms used in \cite{cedo2014braces}. Given two left braces $A$ and $H$ and a homomorphism $\alpha:A\rightarrow \Aut(H)$ from the group $(A,\circ)$ to the automorphisms of the left brace $H$, the \emph{semidirect product of $A$ and $H$ via $\alpha$} is the left brace $(A\times H, \ +,\ \circ)$, where the sum is the direct sum of the groups $(A,+)$ and $(H,+)$ and the multiplication is the semidirect product of the multiplicative groups $(A,\circ)$ and $(H,\circ)$ via $\alpha$.
\medskip

Let $A$ be a left brace, $H$ a trivial left brace, and $X$ (resp. $S$) a cycle base of $A$ (resp. $H$). By the results provided in \cite{Ru08,rump2006modules}, one can easily show that semidirect products of $A$ and $H$ and semidirect products of $X$ and $S$ are ``compatible'' in the following sense: if $A\ltimes_{\alpha} H$ is a semidirect product such that $\alpha_a(S)=S$, for every $a\in A$, then $\alpha$ induces a semidirect product of the cycle sets $X$ and $S$; on the other side, we have that if $X$ is a cycle set and $S$ a trivial cycle set, then a semidirect product of the cycle sets $X$ and $S$ induces a semidirect product of the left braces $G_X$ and $\mathbb{Z}^S$, where $G_X$ is the left brace 
having as additive group the free abelian group $\mathbb{Z}^X$ and as multiplicative group the one generated by the elements of $X$ subject to the relations $x*y=\sigma_x^{-1}(y)*(\sigma_x^{-1}(y)\cdot x)$, for all $x,y\in X$,
and $\mathbb{Z}^S$ 
is the trivial left brace on the free abelian group $\mathbb{Z}^S$.
\medskip

In the previous section, we showed that to obtain a constant orthogonal dynamical extension $X\times_{\alpha} S$, it is required the transitivity of a subgroup $H_y$ of $\mathcal{G}(X\times_{\alpha} S)$ on the set $\{y\}\times S$. 
Now, using the embedding of $X$ into a left brace $A$ as transitive cycle base, we show that the transitivity of $H_y$ on $\{y\}\times S$ can be carried into the richer structure of the left brace $A$.

\begin{theor}\label{th:ort-dynamical-ex}
Let A be a left brace, $S$ a finite set and $\bar{\alpha}:A\longrightarrow \Sym(S)$ an homomorphism from $(A,\circ)$ to $\Sym(S)$. Moreover, let $X$ be a finite transitive cycle base of $A$ and $e$ an element of $X$ and suppose that the stabilizer $A_e$ of $e$ in $A$ (by the map $\lambda$) acts transitively on $S$ by $\bar{\alpha}$. Then, the binary operation $\cdot$ on $X\times S$ given by
$$(x,s)\cdot (y,t):=(\lambda^{-1}_x(y),\bar{\alpha}^{-1}_x(t)) $$
makes $X\times S$ into a constant orthogonal dynamical extension
of $X$ by $\bar{\alpha}^{-1}_{|_X}$.\\
Conversely, any constant orthogonal dynamical extension can be constructed in this way.
\end{theor} 
\begin{proof}
Clearly, the binary operation given by $x\cdot y:=\lambda^{-1}_x(y)$ makes $X$ into a cycle set. Moreover, we have that
\begin{align*}
\bar{\alpha}_{x+y}=\bar{\alpha}_{y+x} 
&\Rightarrow \bar{\alpha}_{x\circ \lambda^{-1}_x(y)}=\bar{\alpha}_{y\circ \lambda^{-1}_y(x)}\\
&\Rightarrow \bar{\alpha}^{-1}_{ \lambda^{-1}_x(y)} \bar{\alpha}^{-1}_{x}=\bar{\alpha}^{-1}_{ \lambda^{-1}_y(x)} \bar{\alpha}^{-1}_{y} 
\end{align*}
for all $x,y\in X$, hence $(X\times S,\cdot)$ is a cycle set. Clearly, it is a constant dynamical extension of $X$ by $S$ having dynamical cocycle $\alpha:=\bar{\alpha}^{-1}_{|_X}$.\\
Since $X$ is a transitive cycle base of $A$, by \cite[Proposition 9]{rump2020one} $X$ is an indecomposable cycle set. Therefore, it remains to show that the subgroup $H_e$ defined as in \cref{carindec} acts transitively on $\{e\}\times S$. Let $s_1$ and $s_2$ be elements of $S$. Since $A_e$ acts transitively on $S$ by $\bar{\alpha}$, there exists $a_e\in A_e$ such that $\bar{\alpha}(a_e)(s_1)=s_2$. 
Moreover, by \cite[Proposition 8]{rump2020one}, there exist $g_1,\ldots,g_n\in X$, $\epsilon_1,\ldots,\epsilon_n\in \{-1,1\}$  such that $b_e=g_1^{\epsilon_1}\circ\cdots\circ g_n^{\epsilon_n}$. Then, it follows that
\begin{align*}
    \bar{\alpha}(a_e)(s_1)=s_2 &\Rightarrow \bar{\alpha}(g_1^{\epsilon_1}\circ\cdots\circ g_n^{\epsilon_n})(s_1)=s_2\\ 
&\Rightarrow \bar{\alpha}_{g_1^{\epsilon_1}}\cdots  \bar{\alpha}_{g_n^{\epsilon_n}}(s_1)=s_2\\     
&\Rightarrow  \bar{\alpha}^{\epsilon_1}_{g_1} \cdots \bar{\alpha}^{\epsilon_n}_{g_n}(s_1)=s_2 
\end{align*}
and since $e=\lambda_{a_e}(e) = \lambda^{\epsilon_1}_{g_1} \cdots  \lambda^{\epsilon_n}_{g_n}(e)$ it follows that $H_e$ acts transitively on $\{e\}\times S$, hence $X\times S$ is indecomposable. Clearly, $\{X\times \{s\}\}_{s\in S}$ is a system of blocks orthogonal to $\mathcal{B}_X$.\\
Conversely, suppose that $X\times_{\alpha} S$ is a constant orthogonal dynamical extension. By \cite[Proposition 9]{rump2020one}, $X$ is a transitive cycle base of the left brace $G_X$ having as additive group the free abelian group $\mathbb{Z}^X$ and as multiplicative group the group generated by the elements of $X$ subject to the relation $x*y=\sigma_x^{-1}(y)*(\sigma_x^{-1}(y)\cdot x)$, for all $x,y\in X$. 
Then, since $\alpha_{x\cdot y}\alpha_x=\alpha_{y\cdot x}\alpha_y$, for all $x,y\in X$, replacing $y$ by $\sigma_x^{-1}(y)$ we obtain that $\alpha_{y}\alpha_x = \alpha_{\sigma_x^{-1}(y)\cdot x}\alpha_{\sigma_x^{-1}(y)}$ and hence $\alpha^{-1}_{x}\alpha^{-1}_y = \alpha^{-1}_{\sigma_x^{-1}(y)}\alpha^{-1}_{\sigma_x^{-1}(y)\cdot x} $. Therefore, the assignment $x\mapsto \alpha^{-1}_x$ can be extended to an action $\bar{\alpha}$ from the multiplicative group of $G_X$ to $\Sym(S)$. 
Using \cref{carindec}, by similar computations seen in the previous implication, we obtain that the stabilizer $(G_X)_e$ of an element $e\in X$ acts transitively on $S$ by $\bar{\alpha}$. Moreover, the indecomposable cycle set on $X\times S$ constructed starting from $G_X$ and $\bar{\alpha}$ as in the previous implication is exactly the dynamical extension $X\times_{\alpha} S$. Hence the proof is complete.
\end{proof}

\begin{rems}\hspace{1mm}
\begin{enumerate}
    \item By \cref{th:ort-dynamical-ex}, the classification of constant orthogonal dynamical extensions of a finite indecomposable cycle set $X$ reduces to the classification of the actions $\bar{\alpha}:A\longrightarrow \Sym(S)$, where $S$ is a finite set and $A$ is a left brace having $X$ as transitive cycle base and such that the stabilizer of an element $e\in X$ (by maps $\lambda)$ acts transitively on $S$ by $\bar{\alpha}$.
    \item An action of the multiplicative group $(A,\circ)$ of a left brace $A$ on a set $S$ is a particular case of an action of a skew left brace. This algebraic object has been recently introduced and studied by De Commer in \cite{de2019actions}.
\end{enumerate}
\end{rems}
\medskip

In the following result, we obtain that left braces are a powerful tools to describe constant orthogonal dynamical extensions and to check them by semidirect products.

\begin{cor}\label{th:teor20}
Let A be a left brace having $X$ as a finite transitive cycle base, $H$ a trivial left brace having $S$ as a finite cycle base. 
Assume that 
$\alpha:A\rightarrow \Aut(H)$ 
is a homomorphism from the group $(A,\circ)$ to 
the automorphisms of the left brace $H$ such that $\alpha(a)(S)=S$, for every $a\in A$, and there exists $e\in X$ such that the stabilizer $A_e$ acts transitively on $S$ by $\alpha$. Then, the semidirect product of the left brace $A\ltimes_{\alpha} H$ has the constant dynamical extension $X\times_{\bar{\alpha}} S$ as transitive cycle base, where $\bar{\alpha}:X\rightarrow \Sym(S)$ is given by $\bar{\alpha}(x):=\alpha^{-1}_{}(x)$, for every $x\in X$. Moreover, $X\times_{\bar{\alpha}} S$ is a constant orthogonal dynamical extension.\\
Conversely, every constant orthogonal dynamical extension can be constructed in this way.
\end{cor}

\begin{proof}
Note that the homomorphism $\alpha:A\rightarrow \Aut(H,+)$ induces a homomorphism $\alpha^*:A\rightarrow \Sym(S)$ such that $A_e$ acts transitively on $S$ by $\alpha^*$. Then, by  \cref{th:ort-dynamical-ex}, $X\times_{\bar{\alpha}} S$ is a constant orthogonal dynamical extension. Since 
the additive group of the left brace $A\ltimes_{\alpha} H$ is 
the direct sum of $(A,+)$ and $(H,+)$, it follows that $X\times_{\bar{\alpha}} S$ is a transitive cycle base.\\
For the converse part, given a constant orthogonal dynamical extension $X\times_{\alpha} S$, it is sufficient to consider $X$ as a transitive cycle base of the left brace $G_X$ and $S$ as the cycle base of the trivial left brace over the free abelian group $\mathbb{Z}^{S}$. In this way, by \cref{th:ort-dynamical-ex},
the function $\alpha$ induces an action $\bar{\alpha}$ of $G_X$ on $S$. Again, $\bar{\alpha}$ induces an action $\bar{\bar{\alpha}}$ by automorphisms of $G_X$ on the free abelian group $\mathbb{Z}^{S}$. Finally, the semidirect product of left braces $G_X\ltimes_{\bar{\bar{\alpha}}} \mathbb{Z}^{S}$ has the dynamical extension $X\times_{\alpha} S$ as a transitive cycle base.
\end{proof}
\smallskip

We conclude this section by giving an example of constant dynamical extensions obtained by using \cref{th:teor20}.

\begin{ex}
Let $A$ be the left brace having $(\mathbb{Z}/4\mathbb{Z},+)$ as additive group and with multiplication given by $a\circ b:=a+b+2ab$, for all $a,b\in A$, and let $H$ be the trivial left brace on $(\mathbb{Z}/3\mathbb{Z},+)$. 
Consider $G:=A\ltimes_{\alpha} H$ the semidirect product of the left braces $A$ and $H$ where $\alpha:A\longrightarrow \Aut(H)$ is the homomorphism from $(A,\circ)$ to the automorphism group of the left brace $H$ given by $\alpha(0)=\alpha(3):=\id_H$ and $\alpha(1)=\alpha(2):=(1\;2)$. Then, $X:=\{1,3\}$ is the unique transitive cycle base of $A$ and $S:=\{1,2\}$ is a cycle base of $H$. Moreover, the stabilizer $A_1:=\{0,2\}$ of $1$ is a subgroup of $(A,\circ)$ that acts transitively on $S$. Hence, if $\bar{\alpha}$ is the function defined as in \cref{th:teor20}, it follows that $X\times_{\bar{\alpha}} S$ is a constant orthogonal dynamical extension. In particular, it is isomorphic to the that on $\mathbb{Z}/2\mathbb{Z}\times \mathbb{Z}/2\mathbb{Z}$ given by 
$$
(x,s)\cdot (y,t):=(y+1,t+x),
$$
for all $(x,s),(y,t)\in \mathbb{Z}/2\mathbb{Z}\times \mathbb{Z}/2\mathbb{Z}$.
\end{ex}

\section{Orthogonal dynamical extensions: examples and constructions}

In this section, we specialize \cref{carindec} to provide examples and constructions of indecomposable cycle sets, giving special attention to the ones having abelian permutation group. Finally, we show that, by semidirect products of cycle sets, one can construct examples of orthogonal dynamical extensions that are not constant.
\medskip

Let us recall that, if $X$ is a finite indecomposable cycle set with abelian permutation group and $A$ is a finite abelian group, the constant orthogonal dynamical extensions having permutation group isomorphic to $\mathcal{G}(X)\times A$ simply correspond to the maps $\alpha:X\longrightarrow A$ from $X$ to the translations group of $A$ (which we identify with $A$) such that $\alpha(x\cdot y)+\alpha(x)=\alpha(y\cdot x)+\alpha(y)$, for all $x,y\in X$, and the group 
$$
\mathcal{H}_y:=\{ \epsilon_1 \alpha(x_1)+\cdots+\epsilon_n \alpha(x_n)\ | \  n\in \mathbb{N},\ \epsilon_1,...,\epsilon_n\in \{-1,1\} \quad \sigma^{\epsilon_1}_{x_1}\cdots\sigma^{\epsilon_n}_{x_n}(y)=y\}
$$
is equal to $A$, for some $y\in X$. \\
We underline that the indecomposable cycle sets constructed in this section have not cyclic permutation group, hence they are different from those obtained in \cite{capiru2020,Ru20}.
\smallskip

At first, we introduce the following example that, in particular, coincides with \cite[Example 10]{cacsp2018}.

\begin{ex}
Let $X:=\mathbb{Z}/k\mathbb{Z}$ be the cycle set given by $x\cdot y:=y+1$, for all $x,y\in X$. Moreover, let $S:=\mathbb{Z}/k\mathbb{Z}$ and $\alpha:X\longrightarrow \Sym(S)$ the function given by $\alpha_x:=\mathfrak{t}_x$, for every $x\in X$, where $\mathfrak{t}_x$ is the translation by $x$. Then,
$$
(x,s)\cdot (y,t)=(y+1,\alpha_x(t)),
$$
for all $x,y\in X,$ $s,t\in S$ and $X\times_{\alpha} S$ is an orthogonal dynamical extension of $X$ by $\alpha$ having multipermutation level $2$ and permutation group isomorphic to $\mathbb{Z}/k\mathbb{Z} \times \mathbb{Z}/k\mathbb{Z}$.
\end{ex}



 
 The previous example is an orthogonal dynamical extension $X\times_{\alpha} S$ for which $\Ret(X\times_{\alpha} S)=X$. However, there exist orthogonal dynamical extensions $X\times_{\alpha} S$ for which $\Ret(X\times_{\alpha} S)$ is not isomorphic to $X$, as we can see in the following example.
 
\begin{ex}\label{refret}
Let $X:=\mathbb{Z}/p^2\mathbb{Z}$ be the cycle set given by $x\cdot y:=y+1$, for all $x,y\in X$. Moreover, let $H:=p\mathbb{Z}/p^2\mathbb{Z}$, $S:=X/H$, and $\alpha:X\longrightarrow \Sym(S)$ the function given by $\alpha_x:=\mathfrak{t}_{[x]}$, for every $x\in X$, where $[x]$ is the image of $x$ in $X/H$ under the canonical epimorphism and $\mathfrak{t}_{[x]}$ is the translation by $[x]$. Then,
$$
(x,s)\cdot (y,t)=(y+1,t+[x]),
$$
for all $x,y\in X,$ $s,t\in S$, and $X\times_{\alpha} S$ is a constant orthogonal dynamical extension of $X$ by $\alpha$ having multipermutation level $2$: indeed, $|\Ret(X\times S)|=p$ and $\alpha_{x\cdot y}\alpha_x=\mathfrak{t}_{[x+y+1]} $, which is symmetric with respect to $x$ and $y$. Moreover, $\sigma_{(0,0)}^{p^2-1}\sigma_{(1,0)}\in H_{0}$ and $<\sigma_{(0,0)}^{p^2-1}\sigma_{(1,0)}>$ acts transitively on $\{0\}\times S$. Clearly, the permutation group is isomorphic to $\mathbb{Z}/p^2\mathbb{Z} \times \mathbb{Z}/p\mathbb{Z}$. 
\end{ex}
\medskip

The indecomposable cycle sets given before are of multipermutation level $2$ and with abelian permutation group. This class of cycle sets is studied in detail in \cite{JePiZa20x,capiru2020}. Now, we give further examples of indecomposable cycle sets having greater multipermutation level. 
Such examples are also different from the ones provided in \cite{cacsp2018}: indeed, the multipermutation cycle sets contained in \cite{cacsp2018} have multipermutation level less than $2$.   

\begin{ex}\label{ex:level>2}
Let $X:=\mathbb{Z}/4\mathbb{Z}$ be the cycle set given by $\sigma_0=\sigma_2:=\mathfrak{t}_1$ and $\sigma_1=\sigma_3:=\mathfrak{t}_{-1}$, where $\mathfrak{t}_1$ and $\mathfrak{t}_{-1}$ are the translations by $1$ and $-1$, respectively. Moreover, let $S$ be the group $\mathbb{Z}/2\mathbb{Z}$ and $\alpha:X\longrightarrow \Sym(S)$ given by 
$$
\alpha_0=\alpha_1:=\id_S \qquad \alpha_2=\alpha_3:=\mathfrak{t}_1.
$$
By a long but easy calculation, one can show that $\alpha_{x\cdot y}\alpha_x=\alpha_{y\cdot x}\alpha_y $, for all $x,y\in X$, and that the subgroup $<\sigma^3_{(0,0)} \sigma_{(2,0)}> $ of $H_{0}$ acts transitively on $\{0\}\times S$. Moreover, the function $\theta:X\times S\longrightarrow X$, $(x,s)\mapsto x$ induces an isomorphism from $\Ret(X\times S)$ to $X$. Therefore, $X\times_{\alpha} S$ is a constant orthogonal dynamical extension of $X$ by $S$ of multipermutation level $3$ having permutation group isomorphic to $\mathbb{Z}/4\mathbb{Z}\times \mathbb{Z}/2\mathbb{Z} $. In particular, $X\times_{\alpha} S$ is isomorphic to the cycle set on $\mathbb{Z}/8\mathbb{Z}$ given by
\begin{align*}
    &\sigma_0=\sigma_4:=(0\;1\;2\;3)(4\;5\;6\;7)
    \quad 
    &\sigma_1=\sigma_5:=(0\;3\;2\;1)(4\;7\;6\;5)\\
    &\sigma_2=\sigma_6:=(0\;5\;2\;7)(1\;6\;3\;4) 
    \quad
    &\sigma_3=\sigma_7:=(0\;7\;2\;5)(1\;4\;3\;6).
\end{align*}
\end{ex}

\bigskip

%

In the following proposition, we construct a family of indecomposable cycle sets of multipermutation level $3$ obtained as constant orthogonal dynamical extensions of cycle sets having multipermutation level $2$.

\begin{prop}\label{prop:22}
Let $p$ be an odd prime number and $X:=\mathbb{Z}/p^{2s}\mathbb{Z}$ the indecomposable cycle set of multipermutation level $2$ given by $x\cdot y:=y+1+xp^s$, for all $x,y\in X$. Let $k$ be a number coprime with $p$ and $f$ the function from $X$ into itself given by
$$ f(j):=k(j+\frac{p^s j(j-1)}{2}), $$
for every $j\in X$, where $\frac{p^s j(j-1)}{2}$ is the element $a\in X$ such that $2a=p^s j(j-1)$. Let $\alpha:X\rightarrow \mathbb{Z}/p^{2s}\mathbb{Z}$ be the function given by $\alpha_x:=\mathfrak{t}_{f(x)}$, for all $x\in X$.
Then, $X\times_{\alpha} \mathbb{Z}/p^{2s}\mathbb{Z}$ is a constant  orthogonal dynamical extension of $X$ by $\alpha$  having multipermutation level $3$. Moreover, $\Ret(X\times \mathbb{Z}/p^{2s}\mathbb{Z})\cong X$ and $\mathcal{G}(X\times \mathbb{Z}/p^{2s}\mathbb{Z})\cong \mathbb{Z}/p^{2s}\mathbb{Z}\times \mathbb{Z}/p^{2s}\mathbb{Z}$.
\end{prop}

\begin{proof}
To show that $X\times_{\alpha} \mathbb{Z}/p^{2s}\mathbb{Z}$ is a cycle set, it is sufficient to see that $f(i\cdot j)+f(i)=f(j\cdot i)+f(j)$, for all $i,j\in X$. Let $i,j$ be elements of $X$. Then,
\begin{eqnarray}
& & f(i\cdot j)+f(i) = f(j+1+p^si)+f(i) \nonumber \\
&=& k(j+1+p^si+\frac{p^s(j+1+p^si)(j+1+p^si-1)}{2})+k(i+\frac{p^si(i-1)}{2}) \nonumber \\
&=& k(j+1+p^si+\frac{p^s(j+1)j}{2})+k(i+\frac{p^si(i-1)}{2}) \nonumber \\
&=& k(j+i+1+\frac{p^s(j+1)j}{2}+p^s\frac{(i+1)i}{2}) \nonumber
\end{eqnarray}
and since the expression is symmetric with respect to $i$ and $j$ it follows that $f(i\cdot j)+f(i) = f(j\cdot i)+f(j)$, hence $X\times_{\alpha} \mathbb{Z}/p^{2s}\mathbb{Z}$ is a cycle set. Moreover, the subgroup $<\sigma^{p^{2s}-p^s-1}_{(0,0)}\sigma_{(1,0)} >$ of $H_0$ acts transitively on $\{0\}\times \mathbb{Z}/p^{2s}\mathbb{Z}$, hence, by \cref{carindec} the cycle set $X\times_{\alpha} \mathbb{Z}/p^{2s}\mathbb{Z}$ is indecomposable and $\mathcal{G}(X\times_{\alpha} \mathbb{Z}/p^{2s}\mathbb{Z})$ is isomorphic to $\mathbb{Z}/p^{2s}\mathbb{Z}\times \mathbb{Z}/p^{2s}\mathbb{Z}$.\\
Finally, let $\rho:X\times \mathbb{Z}/p^{2s}\mathbb{Z} \longrightarrow X $ be the function given by $\rho(x,s):=x$, for every $(x,s)\in X\times \mathbb{Z}/p^{2s}\mathbb{Z}$. Clearly, $\rho$ is an epimorphism of cycle sets. Moreover, since $f(j+ap^s)=f(j)+kap^s$, for 
all $j,a\in \mathbb{Z}/p^{2s}\mathbb{Z}$, it follows that $\sigma_{(x,s)}=\sigma_{(y,t)}$ if and only if $\rho(x)=\rho(y)$, therefore $\Ret(X \times_{\alpha} \mathbb{Z}/p^{2s})\cong X$ and $X \times \mathbb{Z}/p^{2s}$ has multipermutation level $3$.
\end{proof}

\medskip


The following proposition shows that other examples of indecomposable cycle sets that are constant dynamical extensions can be easily obtained by those having cyclic permutation group and arbitrary multipermutation level.

\begin{prop}\label{prop:23}
Let $n,k\in \mathbb{N}$, $X:=\{0,\ldots,p^k-1\}$ be an indecomposable cycle set of multpermutation level $n$ and with cyclic permutation group constructed as in \cref{costruz2}. Moreover, let $j\in\mathbb{N}$ such that $|\Ret^{n-1}(X)|=p^j$ and $s$ a natural number in $\{1,\ldots,j\}$. Let $\alpha:X\rightarrow \Sym(\mathbb{Z}/p^s\mathbb{Z})$ be the function given by $\alpha_x:=\mathfrak{t}_{[x]}$, where $[x]$ is the class of $x$ module $p^s$.
Then, $X\times_{\alpha} \mathbb{Z}/p^s\mathbb{Z}$ is a constant orthogonal dynamical extension of $X$ such that $\mathcal{G}(X\times_{\alpha} \mathbb{Z}/p^s\mathbb{Z})\cong \mathbb{Z}/p^k\mathbb{Z}\times \mathbb{Z}/p^s\mathbb{Z} $ and $\Ret(X\times_{\alpha} S)\cong \Ret(X)$.
\end{prop}

\begin{proof}
Similarly to \cref{prop:22}, one can show that $X\times_{\alpha} \mathbb{Z}/p^s\mathbb{Z}$ is an indecomposable cycle set such that $\mathcal{G}(X\times_{\alpha} \mathbb{Z}/p^s\mathbb{Z})\cong \mathbb{Z}/p^k\mathbb{Z}\times \mathbb{Z}/p^s\mathbb{Z} $. Now, we show that $\Ret(X\times_{\alpha} \mathbb{Z}/p^s\mathbb{Z})$ is isomorphic to $\Ret(X)$. Let $\rho:X\times S\longrightarrow \Ret(X)$ be the function given by $\rho(x,s):=\sigma_x$, for all $x\in X,s\in S$. Clearly, $\rho$ is an epimorphism of cycle sets. Moreover,
$$\sigma_x=\sigma_y\Rightarrow |\Ret(X)| \mid y-x\Rightarrow [x]=[y], $$
for all $x,y\in X$, therefore the equivalence relation induced by $\rho$ coincides with the retract relation of $X\times \mathbb{Z}/p^s\mathbb{Z}$, hence the proof is complete.
\end{proof}
\smallskip

\begin{ex}
Let $X:=\mathbb{Z}/32\mathbb{Z}$ be the cycle set of multipermutation level $3$ and cyclic permutation group given by 
$$\sigma_x:= \begin{cases} 
\mathfrak{t}_1 &\mbox{if } x \equiv 0\;(mod\; 8) \\ 
\mathfrak{t}_1^5 & \mbox{if } x \equiv 1\;(mod\; 8)\\ 
\mathfrak{t}_1^{25} & \mbox{if } x \equiv 2\;(mod\; 8)\\ 
\mathfrak{t}_1^{29} & \mbox{if } x \equiv 3\;(mod\; 8)\\ 
\mathfrak{t}_1^{17} & \mbox{if } x \equiv 4\;(mod\; 8)\\ 
\mathfrak{t}_1^{21} & \mbox{if } x \equiv 5\;(mod\; 8)\\ 
\mathfrak{t}_1^9 & \mbox{if } x \equiv 6\;(mod\; 8)\\ 
\mathfrak{t}_1^{13} & \mbox{if } x \equiv 7\;(mod\; 8)\end{cases} $$
and $\alpha:X\longrightarrow \mathbb{Z}/2\mathbb{Z}$ the function given by $\alpha_x:=\mathfrak{t}_{[x]}$, where $[x]$ is the class of $x$ module $2$. By \cref{prop:23}, $X\times_{\alpha} \mathbb{Z}/2\mathbb{Z}$ is a constant orthogonal dynamical extension of $X$ by $\alpha$ having multipermutation level $3$ and permutation group isomorphic to $\mathbb{Z}/32\mathbb{Z}\times \mathbb{Z}/2\mathbb{Z}$.
\end{ex}

\medskip


The following example ensures that the class of constant orthogonal dynamical extensions is not contained in that of cycle sets with abelian permutation group. 
Furthermore, this example of cycle set is different from the previous ones since it is indecomposable and has not finite multipermutation level.

\begin{ex}
Let $X:=\{0,1,2,3\}$ be the irretractable indecomposable cycle set given by $\sigma_0:=(1,2)$, $\sigma_1:=(0,3)$, $\sigma_2:=(0,2,3,1)$, and $\sigma_3:=(0,1,3,2)$. Let $S:=\mathbb{Z}/4\mathbb{Z}$ and $\alpha:X\longrightarrow S$ the function given by $\alpha_0=\alpha_1:=\mathfrak{t}_1$ and $\alpha_2=\alpha_3:=\mathfrak{t}_3$. Then, by \cref{carindec}, the cycle set $X\times_{\alpha} S$ is an indecomposable cycle set of size $16$: indeed, $\alpha_{x\cdot y}\alpha_x$ is symmetric with respect to $x$ and $y$, $X$ is indecomposable and the subgroup of $H_0$ generated by $\sigma_{(0,0)}$ acts transitively on $\{0\}\times S$. Moreover, since $\Ret(X\times_{\alpha} S)$ is isomorphic to $X$, $X\times_{\alpha} S$ has not finite multipermutation level.  Since $\mathcal{G}(X)$ is not abelian, $\mathcal{G}(X\times_{\alpha} S) $ is not.
\end{ex}

Providing examples of orthogonal dynamical extensions that are not constant is rather difficult. However, the semidirect product of cycle sets belonging to the class of quasigroups allows for obtaining a family of examples of orthogonal dynamical extensions. We recall that a \emph{quasigroup} is an algebraic structure for which left and right multiplications are bijective, see \cite{bon2019}.

\begin{prop}\label{prop:semidirect-quasi}
Let $X,S$ be finite cycle sets and $\alpha:X\rightarrow \Aut(S)$ a function from $X$ to $\Aut(S)$ such that $\alpha_{x\cdot y}\alpha_x=\alpha_{y\cdot x}\alpha_y$, for all $x,y\in X$. Then the semidirect product of $X$ and $S$ via $\alpha$ is a quasigroup if and only if $X$ and $S$ are quasigroups. Moreover, if $X\ltimes_{\alpha} S$ is a quasigroup, then it is an orthogonal dynamical extension.
\end{prop}

\begin{proof}
It is a routine computation to check that $X\ltimes_{\alpha} S$ is a quasigroup if and only if $X$ and $S$ are quasigroups. Clearly, every cycle set that also is a quasigroup is indecomposable. Therefore, if $X\ltimes_{\alpha} S$ is a quasigroup, by \cref{rema}, it is an orthogonal dynamical extension.
\end{proof}

\smallskip

\begin{rems}
\hspace{1mm}
\begin{enumerate}
    \item To construct examples of orthogonal dynamical extensions using \cref{prop:semidirect-quasi}, one has to provide cycle sets that also are quasigroups. Even if these cycle sets are not easy to compute, several instances have been provided by Bonatto, Kinyon, Stanovsk\'{y}, and Vojt\v{e}chovsk\'{y} in \cite{bon2019}. For an exhaustive study of these structures, we refer the reader to the same paper \cite{bon2019}.
    \item The indecomposable cycle sets given in \cite{capiru2020,Ru20,cacsp2018, JePiZa20x} are not quasigroups: in this sense, \cref{prop:semidirect-quasi} provides different examples of indecomposable cycle sets.
    \item \cref{prop:semidirect-quasi} allows for determining ``genuine'' examples of orthogonal dynamical extension. Indeed, if $X\ltimes_{\alpha} S$ is a quasigroup, then it clearly is an irretractable cycle set, hence it can not be constructed by a constant orthogonal dynamical extension. 
\end{enumerate}
\end{rems}

\smallskip

We close this section by showing a concrete example of orthogonal dynamical extension which we construct using \cref{prop:semidirect-quasi}.

\begin{ex}
Let $X:=\{0,1,2,3\}$ be the indecomposable cycle set given by $\sigma_0:=(1,2)$, $\sigma_1:=(0,3)$, $\sigma_2:=(0,2,3,1)$, and $\sigma_3:=(0,1,3,2)$. Moreover, let $(S,\cdot)$ be a copy of $X$ and $\alpha:X\rightarrow \Aut(S)$ the function given by $\alpha_0=\alpha_1:=(0,1)(2,3)$ and $\alpha_2=\alpha_3:=\id_S$. By a long but easy calculation, it follows that $X$ and $S$ are quasigroups and $\alpha_{x\cdot y}\alpha_x=\alpha_{y\cdot x}\alpha_y$, for all $x,y\in X$. 
Therefore, $X\ltimes_{\alpha} S$ is an orthogonal dynamical extension.
\end{ex}

\section{One-generator left braces of multipermutation level 2}

In this section, which is self-contained with respect to the previous ones, we study the class of one-generator left braces, which is closely related to indecomposable cycle sets (see \cite{smock,rump2020one} for more details). Specifically, we investigate one-generator left braces of multipermutation level at most $2$ focusing our attention on the properties related to the map $\lambda$. In this way, we will prove \cite[Theorem 2]{rump2020one} by a different proof. At this purpose, let us recall that a left brace $A$ is said to be a \textit{one-generator left brace} if there exists $x\in A$ such that $A = A(x)$, where $A(x)$ is the smallest left sub-brace of $A$ containing $x$.

\medskip 

At first, for any subset $S$ of a left brace $A$, let us put
\begin{align*}
    - S:= \{-s \ | \ s\in S\}
    \qquad\text{and}\qquad
    S^{-}:= \{s^- \ | \ s\in S\}.
\end{align*}

\begin{lemma}
Let $A$ be a left brace and $x\in A$. Define inductively 
$A_0:=\{x,\,-x,x^-, 0\}$
and
$$
A_i = (A_{i-1} + A_{i-1})\cup (A_{i-1} - A_{i-1})\cup (A_{i-1}\circ A_{i-1}) \cup (A_{i-1}\circ A_{i-1}^-)  \cup (A_{i-1}^-\circ A_{i-1}),
$$
for every $i\in \mathbb{N}$.
Then, $\bigcup\limits_{n\in \mathbb{N}_0} A_n = A(x)$.
\end{lemma}

\begin{proof}
Let $U:=\bigcup\limits_{n\in \mathbb{N}_0} A_n$. At first, $U$ is closed under the operations $+$ and $\circ$. 
Indeed, if $a\in A_s$ and $b\in A_t$, since $A_k\subseteq A_{k+j}$, for all $k,j\in \mathbb{N}$, $a+b,a\circ b\in A_{s+t+1}$. Moreover, $0\in U$, hence if $a\in A_s$ we have that $a^-=0\circ a^-\in A_{s+1}$ and, similarly, $-a=0-a\in A_{s+1}$. 
Therefore, $U$ is a left brace containing $x$.\\
Now, if $B$ is a left sub-brace of $A$ containing $x$, by induction on $i$, it is easy to show that $A_i\subseteq B$, hence $U\subseteq B$. Therefore, $U$ is equal to $A(x)$. 
\end{proof}



\medskip

The following result is originally contained in \cite[Theorem 9.2]{JePiZa}. 
In particular, it shows that the map $\lambda$ satisfies an additional property with respect to \cref{action} if the left brace has multipermutation level at most $2$.
\begin{lemma}\label{add}
Let $A$ be a left brace. Then the map $\lambda:A\longrightarrow \Aut(A,+)$ is a homomorphism from $(A,+)$ into $\Aut(A,+)$ if and only if $A$ is a trivial left brace or $A$ has multipermutation level $2$.
\end{lemma}

Following \cite{baneya}, a left brace $A$ is said to be $\lambda$-\textit{cyclic} if $\lambda(A)$ is a cyclic group. In the following proposition, we show that one-generator left braces of multipermutation level $2$ are always $\lambda$-cyclic.

\begin{prop}\label{act}
Let $A$ be a one-generator left brace of multipermutation level $2$, $x\in A$, and suppose that $A=A(x)$. Then, the subgroup generated by $\lambda(A)$ is equal to the one generated by $\lambda_x$.
\end{prop}

\begin{proof}
It is sufficient to show that $\lambda_a=\lambda_x^{k}$, for some $k\in \mathbb{Z}$, for every $a\in A_i$, proceeding by induction on $i$. Since $A_0=\{0,x,-x,x^-\}$, by \cref{add} we have that $\id_A=\lambda_{x-x}=\lambda_x\lambda_{-x}$, hence $\lambda_{-x}=\lambda_x^{-1}$ and, by \cref{action}, $\lambda_{x^-}=\lambda_x^{-1}$. 
Now, suppose that $\lambda_a=\lambda_x^{k}$, for some $k\in \mathbb{Z}$, for every $a\in A_i$, and let $a_1,a_2$ be elements of $A_i$ and $k_1,k_2\in \mathbb{Z}$ such that $\lambda_{a_1}=\lambda_x^{k_1}$ and $\lambda_{a_2}=\lambda_x^{k_2}$. Then, by \cref{add}, $\lambda_{a_1+a_2}=\lambda_x^{k_1+k_2}$ and $\lambda_{a_1-a_2}=\lambda_x^{k_1-k_2}$. By \cref{action}, we have that $\lambda_{a_1\circ a_2}=\lambda_x^{k_1+k_2}$ and $\lambda_{a_1\circ a_2^-}=\lambda_{a_2^-\circ a_1}=\lambda_x^{k_1-k_2}$, hence the claim follows.
\end{proof}

\medskip

In the following lemma, we provide a necessary and sufficient condition that characterizes the trivial one-generator left braces. This result will be of crucial importance to give a precise description of one-generator left braces of multipermutation level $2$.

\begin{lemma}\label{carbra}
Let $A$ be a one-generator left brace and $x\in A$ such that $A = A(x)$. Then, the following statements are equivalent:
\begin{itemize}
    \item[1)] the map $\lambda_x$ fixes $x$;
    \item[2)] $A$ is a trivial left brace with $(A,+)$ the additive subgroup generated by $x$.
\end{itemize}
\end{lemma}
\begin{proof}
Suppose that $1)$ holds. At first, we denote by $ka$ the element of $A$ given by $ka:=\underbrace{a+ \cdots +a}_{k\ times}$, for all $a\in A$, $k\in\mathbb{N}$. Similarly, let $a^{\circ k}$ be the element of $A$ given by $a^{\circ k}:=\underbrace{a\circ\cdots \circ a}_{k\ times}$. 
Since $\lambda_x$ fixes $x$, it is easy to show by induction on $k$ that $kx=x^{\circ k}$ and $k(-x)=(-x)^{\circ k} =(x^-)^{\circ k}$, for every $k\in \mathbb{N}$. In this way, we obtain that $kx\circ hx=kx+\lambda_{kx}(hx)=kx+hx$, for all $k,h\in \mathbb{Z}$, hence the additive subgroup generated by $x$ is a trivial left brace containing $x$. Since $A=A(x)$, we obtain  that $1)$ implies $2)$. The converse implication is trivial. 
\end{proof}

Now, we are able to prove a characterization of one-generator left braces of multipermutation level $2$ by means of transitive cycle bases. 

\begin{theor}\label{th:one-g-2}
Let $A$ be a left brace. Then, the following conditions are equivalent:
\begin{itemize}
    \item[1)] $A$ is a one-generator left brace of multipermutation level $2$;
    \item[2)] $A$ has a transitive cycle base $X$, with $|X|>1$, that is an indecomposable cycle set having multipermutation level $1$.
\end{itemize}
\end{theor}

\begin{proof}
Suppose that $A$ is a one-generator left brace of multipermutation level $2$ and let $x$ be such that $A = A(x)$. By \cref{act}, the orbit of $x$ under the action of $A$ by $\lambda$ is equal to $X:=\{\lambda_x^k(x)\ | \ k\in \mathbb{Z} \} $. Moreover, since $A$ has multipermutation level $2$, by \cref{carbra} it follows that $|X|>1$. By \cite[Proposition 6.4]{CeSmVe19}, the additive subgroup generated by $X$ is equal to $A$, hence $X$ is a transitive cycle base. Moreover, $X$ is an indecomposable cycle set having permutation group equal to the cyclic subgroup generated by $\lambda_{x_{|_{X}}}$. Since $X$ is a cycle base, we have that $\lambda_x=\lambda_y$ if and only if $\lambda_{x_{|_{X}}}=\lambda_
{y_{|_{X}}} $, hence we can identify $\Ret(X)$ with $\{\lambda_x^k(x) + \Soc(A) \ |\ k\in \mathbb{Z} \}$. Since $A$ has multipermutation level $2$, we have that $\lambda_{\lambda_x^k(x) + \Soc(A)} = \lambda_{x + \Soc(A)} = \id_{A/\Soc(A)}$, for every $k\in \mathbb{N}_0$, and since $\Ret(X) $ is again indecomposable it follows that $|\Ret(X)|=1$, therefore $2)$ follows.

Now, suppose that $2)$ holds and let $x$ be an element of $X$. At first, we show that $$\lambda_{\epsilon_1 x_1 + \cdots +\epsilon_n x_n} 
= \lambda_x^{\ \sum \epsilon_i},
$$
for all $n\in \mathbb{N}$, $\epsilon_1,\ldots,\epsilon_n\in \{1,-1\} $, and $x_1,\ldots,x_n\in X$.\\
For $n = 1$, if $\epsilon_1 = 1$, the claim follows because $X$ has multipermutation level $1$, if $\epsilon_1=-1$ and $x_1\in X$, we obtain that 
$$
\lambda_0=\lambda_{-x_1+x_1}=\lambda_{(-x_1)\circ \lambda_{(-x_1)^-}(x_1)}=\lambda_{-x_1}\lambda_{\lambda_{(-x_1)^-}(x_1)}
$$
and, since $\lambda_{(-x_1)^-}(x_1)\in X$, we have that $\lambda_{\lambda_{(-x_1)^-}(x_1)} =\lambda_x$, hence $\lambda_{-x_1}=\lambda_x^{-1}$.
Now, suppose that $n$ is a natural number greater than $1$. Therefore
\begin{eqnarray}
\lambda_{\epsilon_1 x_1 + \cdots +\epsilon_n x_n} &=& \lambda_{\epsilon_1 x_1 \circ( \lambda_{(\epsilon_1x_1)^-}(\epsilon_2x_2)+\cdots+\lambda_{(\epsilon_1x_1)^-}(\epsilon_nx_n))} \nonumber \\
&=& \lambda_{\epsilon_1 x_1}\lambda_{ \lambda_{(\epsilon_1x_1)^-}(\epsilon_2x_2)+\cdots+\lambda_{(\epsilon_1x_1)^-}(\epsilon_nx_n)} \nonumber \\
&=& \lambda_x^{\ \sum \epsilon_i} \nonumber 
\end{eqnarray}
where $\lambda_{ \lambda_{(\epsilon_1x_1)^-}(\epsilon_2x_2)+\cdots+\lambda_{(\epsilon_1x_1)^-}(\epsilon_n x_n)} = \lambda_x^{\ \sum\limits_{i=1}^{n-1} \epsilon_i}$ and $\lambda_{\epsilon_1 x_1}=\lambda_x^{\epsilon_1}$, by inductive hypothesis. Then, we have that $X=\{\lambda_x^k(x)\,|\  k\in \mathbb{Z} \}$. By induction on $k$, it is easy to show that $\lambda_x^k(x)$ and $\lambda_x^{-k}(x)$ belong to $A(x)$, for every $k\in \mathbb{N}$, hence $A=X_{+}\subseteq A(x)$. To show $1)$, since $|X|>1$,  by \cref{add} it is sufficient to see that $\lambda_{a+b}=\lambda_{a\circ b}$, for all $a,b\in A$.\\
Finally, let $a,b\in A$, $m,n\in \mathbb{N}$ $\epsilon_1,\ldots,\epsilon_m,\eta_1,\ldots,\eta_n\in \{-1,1\} $ and $x_1,\ldots,x_m$, $y_1,\ldots,y_n\in X$ such that $a = \epsilon_1 x_1 + \cdots + \epsilon_m x_m$ and $b =\eta_1 y_1 + \cdots +\eta_n y_n$. Then,
\begin{eqnarray}
\lambda_{a+b} &=& \lambda_{\epsilon_1 x_1 + \cdots + \epsilon_m x_m+\eta_1 y_1 + \cdots +\eta_n y_n} \nonumber \\
&=& \lambda_x^{\sum \epsilon_i+\sum \eta_j} \nonumber \\
&=& \lambda_x^{\sum \epsilon_i}\lambda_x^{\sum \eta_j} \nonumber \\
&=& \lambda_a\lambda_b \nonumber \\
&=& \lambda_{a\circ b}, \nonumber
\end{eqnarray}
hence the claim follows.
\end{proof}
 
\medskip

As a final result of this paper, we give a precise description of one-generator left braces of multipermutation level $2$ answering \cite[Question 5.5]{smock}. 
We point out that such a description has already contained in \cite[Theorem 2]{rump2020one}; however, here we provide a different proof about it.

\begin{cor}
A one-generator left brace of multipermutation level $2$ is equivalent to an abelian group $(A,+)$ and an automorphism $\psi$ of $(A,+)$ such that
\begin{itemize}
\item[1)] there exists an element $x$ of $A$ whose orbit $\{\psi^z(x) \,| \ z\in \mathbb{Z}\}$ has size greater than $1$ and generates $(A,+)$;
\item[2)] if \ $\sum\limits_{i=1}^{k}a_i\psi^{z_i}(x) = 0$ implies \ $\psi^{\sum\limits_{i=1}^{k}a_i} = \id_A$, for all  $k\in \mathbb{N}_0$ and $a_1,\ldots,a_k$, $z_1,\ldots,z_k\in \mathbb{Z}$.
 \end{itemize}
 \end{cor}
\begin{proof}
Suppose that $A$ is a one-generator left brace of multipermutation level $2$. By \cref{th:one-g-2}, there exists a transitive cycle base $X$ with $|X| > 1$. Let $x\in X$ and $\psi:=\lambda_x$. Then, $\psi$ is an automorphism of $(A,+)$ and the additive group generated by $X=\{\psi^z(x)\, | \ z\in \mathbb{Z} \}$ is equal to $A$. 
Now, if $\sum\limits_{i=1}^{k}a_i\psi^{z_i}(x) = 0$, by \cref{th:one-g-2}-$2)$ and  \cref{add}, it follows that $\psi^{\sum\limits_{i=1}^{k}a_i} = \lambda(\sum\limits_{i=1}^{k}a_i\psi^{z_i}(x)) = \id_A$.\\
Conversely, suppose that $(A,+)$ is an abelian group and $\psi$ an automorphism of $A$ such that $1)$ and $2)$ hold and let $x$ be the element of $A$ such that its orbit $X:=\{\psi^z(x)\, | \ z\in \mathbb{Z}\}$ generates $(A,+)$. Now, by $1)$ and $2)$, the map $\lambda':X\to \Aut(A,+)$ given by $\lambda'(t):=\psi$, for every $t\in X$, can be extended to a group homomorphism $\lambda$ from $(A,+)$ into $\Aut(A,+) $. Now, let $\circ$ be the binary operation on $A$ given by $a\circ b=a+\lambda_a(b)$, for all $a,b\in A$. To prove that $(A,+,\circ)$ is a left brace, by \cite[Proposition 2]{rump2007braces} it remains to show that $\lambda(a\circ b)=\lambda_a\lambda_b$, for all $a,b\in A$. Therefore, let $a,b\in A$ and suppose that $a=\sum a_i \psi^{z_i}(x)$ and $b=\sum b_i \psi^{z_i}(x)$. 
Then,
\begin{eqnarray}
\lambda(a\circ b) &=& \lambda(a+\lambda_a(b)) \nonumber \\
&=& \lambda(a)\lambda(\lambda_a(b)) \nonumber \\
&=& \lambda(a)\lambda(\lambda_a(\sum b_i \psi^{z_i}(x)))\nonumber \\
&=& \lambda(a)\lambda(\sum b_i \lambda_a(\psi^{z_i}(x))) \nonumber \\
&=& \lambda(a)\psi^{\sum b_i} \nonumber \\
&=& \lambda(a)\lambda(b),\nonumber 
\end{eqnarray}
hence $(A,+,\circ)$ is a left brace. \\
Finally, since $X$ generates $A$ and $\lambda(s)=\psi$, for every $s\in X$, by $2)$ of the previous theorem, we have that $A$ is a one-generator left brace of multipermutation level $2$.
\end{proof}

\smallskip


\bibliographystyle{elsart-num-sort}
\bibliography{Bibliography}

\def\cprime{$'$}
\begin{thebibliography}{10}
\expandafter\ifx\csname url\endcsname\relax
  \def\url#1{\texttt{#1}}\fi
\expandafter\ifx\csname urlprefix\endcsname\relax\def\urlprefix{URL }\fi

\bibitem{bachiller2015extensions}
D.~Bachiller, Extensions, matched products, and simple braces, J. Pure Appl.
  Algebra 222~(7) (2018) 1670--1691.
\newline\urlprefix\url{https://doi.org/10.1016/j.jpaa.2017.07.017}

\bibitem{baneya}
V.~G. Bardakov, M.~V. Neshchadim, M.~K. Yadav, On $\lambda$-homomorphic skew
  braces, Preprint arXiv:2004.05555.
\newline\urlprefix\url{https://arxiv.org/abs/2004.05555}

\bibitem{bon2019}
M.~Bonatto, M.~Kinyon, D.~Stanovsk\'{y}, P.~Vojt\v{e}chovsk\'{y}, Involutive
  latin solutions of the {Y}ang-{B}axter equation, J. Algebra 565 (2021)
  128--159.
\newline\urlprefix\url{https://doi.org/10.1016/j.jalgebra.2020.09.001}

\bibitem{cacsp2017}
M.~Castelli, F.~Catino, G.~Pinto, A new family of set-theoretic solutions of
  the {Yang-Baxter} equation, Comm. Algebra 46~(4) (2017) 1622--1629.
\newline\urlprefix\url{http://dx.doi.org/10.1080/00927872.2017.1350700}

\bibitem{cacsp2018}
M.~Castelli, F.~Catino, G.~Pinto, Indecomposable involutive set-theoretic
  solutions of the {Yang-Baxter} equation, J. Pure Appl. Algebra 220~(10)
  (2019) 4477--4493.
\newline\urlprefix\url{https://doi.org/10.1016/j.jpaa.2019.01.017}

\bibitem{capiru2020}
M.~Castelli, G.~Pinto, W.~Rump, On the indecomposable involutive set-theoretic
  solutions of the {Y}ang-{B}axter equation of prime-power size, Communications
  in Algebra 48~(5) (2020) 1941--1955.
\newline\urlprefix\url{https://doi.org/10.1080/00927872.2019.1710163}

\bibitem{cedo2014braces}
F.~Ced{\'o}, E.~Jespers, J.~Okni{\'n}ski, Braces and the {Yang-Baxter}
  equation, Comm. Math. Phys. 327~(1) (2014) 101--116.
\newline\urlprefix\url{https://doi.org/10.1007/s00220-014-1935-y}

\bibitem{CeJeOk20x}
F.~Ced\'o, E.~Jespers, J.~Okni\'nski, Primitive set-theoretic solutions of the
  {Y}ang-{B}axter equation, Preprint arXiv:2003.01983.
\newline\urlprefix\url{https://arxiv.org/abs/2003.01983}

\bibitem{CeSmVe19}
F.~Ced\'{o}, A.~Smoktunowicz, L.~Vendramin, Skew left braces of nilpotent type,
  Proc. Lond. Math. Soc. (3) 118~(6) (2019) 1367--1392.
\newline\urlprefix\url{https://doi.org/10.1112/plms.12209}

\bibitem{chouraqui2010garside}
F.~Chouraqui, Garside groups and {Yang-Baxter} equation, Comm. Algebra 38~(12)
  (2010) 4441--4460.
\newline\urlprefix\url{http://dx.doi.org/10.1080/00927870903386502}

\bibitem{de2019actions}
K.~De~Commer, Actions of skew braces and set-theoretic solutions of the
  reflection equation, Proc. Edinb. Math. Soc. (2) 62~(4) (2019) 1089--1113.
\newline\urlprefix\url{https://doi.org/10.1017/s0013091519000129}

\bibitem{dixon1996permutation}
J.~D. Dixon, B.~Mortimer, Permutation groups, vol. 163, Springer Science \&
  Business Media, 1996.

\bibitem{Do06}
E.~Dobson, Automorphism groups of metacirculant graphs of order a product of
  two distinct primes, Combin. Probab. Comput. 15~(1-2) (2006) 105--130.
\newline\urlprefix\url{https://doi.org/10.1017/S0963548305007066}

\bibitem{DoKo09}
E.~Dobson, I.~Kov\'{a}cs, Automorphism groups of {C}ayley digraphs of
  {$\mathbb{Z}^3_p$}, Electron. J. Combin. 16~(1) (2009) Research Paper 149,
  20.
\newline\urlprefix\url{http://www.combinatorics.org/Volume_16/Abstracts/v16i1r149.html}

\bibitem{drinfeld1992some}
V.~G. Drinfel\cprime~d, On some unsolved problems in quantum group theory, in:
  Quantum groups ({L}eningrad, 1990), vol. 1510 of Lecture Notes in Math.,
  Springer, Berlin, 1992, pp. 1--8.
\newline\urlprefix\url{https://doi.org/10.1007/BFb0101175}

\bibitem{etingof1998set}
P.~Etingof, T.~Schedler, A.~Soloviev, Set-theoretical solutions to the {Quantum
  Yang-Baxter} equation, Duke Math. J. 100~(2) (1999) 169--209.
\newline\urlprefix\url{http://doi.org/10.1215/S0012-7094-99-10007-X}

\bibitem{gateva1998semigroups}
T.~Gateva-Ivanova, M.~Van~den Bergh, Semigroups of {I-Type}, J. Algebra 206~(1)
  (1998) 97--112.
\newline\urlprefix\url{https://doi.org/10.1006/jabr.1997.7399}

\bibitem{JePiZa20x}
P.~Jedli{\v{c}}ka, A.~Pilitowska, A.~Zamojska-Dzienio, Indecomposable
  involutive solutions of the {Y}ang-{B}axter equation of multipermutational
  level $2$ with abelian permutation group, Preprint arXiv:2011.00229.
\newline\urlprefix\url{https://arxiv.org/abs/2011.00229}

\bibitem{JePiZa}
P.~Jedli\v{c}ka, A.~Pilitowska, A.~Zamojska-Dzienio, The construction of
  multipermutation solutions of the {Y}ang-{B}axter equation of level 2, J.
  Combin. Theory Ser. A 176 (2020) 105295, 35.
\newline\urlprefix\url{https://doi.org/10.1016/j.jcta.2020.105295}

\bibitem{JeLeRuVe19}
E.~Jespers, V.~Lebed, W.~Rump, L.~Vendramin, {M}ini-{W}orkshop: {A}lgebraic
  {T}ools for {S}olving the {Y}ang-{B}axter {E}quation, {R}eport no. 51/2019,
  DOI: 10.4171/OWR/2019/51.

\bibitem{lebed2017homology}
V.~Lebed, L.~Vendramin, Homology of left non-degenerate set-theoretic solutions
  to the {Yang-Baxter} equation, Adv. Math. 304 (2017) 1219--1261.
\newline\urlprefix\url{https://doi.org/10.1142/S0218196716500570}

\bibitem{Lu91}
A.~Lucchini, On imprimitive groups with small degree, Rend. Sem. Mat. Univ.
  Padova 86 (1991) 131--142 (1992).
\newline\urlprefix\url{http://www.numdam.org/item?id=RSMUP_1991__86__131_0}

\bibitem{pinto}
G.~Pinto, On the indecomposable cycle sets and the {Y}ang-{B}axter equation,
  PhD Thesis, University of Salento, Italy (2019).

\bibitem{rump2005decomposition}
W.~Rump, A decomposition theorem for square-free unitary solutions of the
  quantum {Yang-Baxter} equation, Adv. Math. 193 (2005) 40--55.
\newline\urlprefix\url{https://doi.org/10.1016/j.aim.2004.03.019}

\bibitem{rump2006modules}
W.~Rump, Modules over braces, Algebra Discrete Math. -~(2) (2006) 127--137.
\newline\urlprefix\url{http://admjournal.luguniv.edu.ua/index.php/adm/article/view/892/421}

\bibitem{rump2007braces}
W.~Rump, Braces, radical rings, and the quantum {Y}ang-{B}axter equation, J.
  Algebra 307~(1) (2007) 153--170.
\newline\urlprefix\url{https://doi.org/10.1016/j.jalgebra.2006.03.040}

\bibitem{Ru08}
W.~Rump, Semidirect products in algebraic logic and solutions of the quantum
  {Y}ang-{B}axter equation, J. Algebra Appl. 7~(4) (2008) 471--490.
\newline\urlprefix\url{https://doi.org/10.1142/S0219498808002904}

\bibitem{rump2016quasi}
W.~Rump, Quasi-linear cycle sets and the retraction problem for set-theoretic
  solutions of the quantum {Y}ang-{B}axter equation, Algebra Colloq. 23~(1)
  (2016) 149--166.
\newline\urlprefix\url{https://doi.org/10.1142/S1005386716000183}

\bibitem{rump2020}
W.~Rump, Classification of indecomposable involutive set-theoretic solutions to
  the {Y}ang-{B}axter equation, Forum Math. 32~(4) (2020) 891--903.
\newline\urlprefix\url{https://doi.org/10.1515/forum-2019-0274}

\bibitem{Ru20}
W.~Rump, Cocyclic solutions to the {Y}ang-{B}axter equation, Proc. Amer. Math.
  Soc. (2020) In press.
\newline\urlprefix\url{https://doi.org/10.1090/proc/15220}

\bibitem{rump2020one}
W.~Rump, One-generator braces and indecomposable set-theoretic solutions to the
  {Y}ang–{B}axter equation, Proc. Edinb. Math. Soc. (2020) 1–21.
\newline\urlprefix\url{https://doi.org/10.1017/S0013091520000073}

\bibitem{Sm18}
A.~Smoktunowicz, On {E}ngel groups, nilpotent groups, rings, braces and the
  {Y}ang-{B}axter equation, Trans. Amer. Math. Soc. 370~(9) (2018) 6535--6564.
\newline\urlprefix\url{https://doi.org/10.1090/tran/7179}

\bibitem{smock}
A.~Smoktunowicz, A.~Smoktunowicz, Set-theoretic solutions of the
  {Y}ang-{B}axter equation and new classes of {R}-matrices, Linear Algebra
  Appl. 546 (2018) 86--114.
\newline\urlprefix\url{https://doi.org/10.1016/j.laa.2018.02.001}

\bibitem{vendramin2016extensions}
L.~Vendramin, Extensions of set-theoretic solutions of the {Yang-Baxter}
  equation and a conjecture of {Gateva-Ivanova}, J. Pure Appl. Algebra 220
  (2016) 2064--2076.
\newline\urlprefix\url{https://doi.org/10.1142/S1005386716000183}

\end{thebibliography}

\end{document}